\documentclass[final,11pt]{imsart}

\arxiv{1706.04059v3}

\usepackage[utf8]{inputenc}
\usepackage[english]{babel}
\usepackage{textcomp}
\usepackage{dsfont}
\usepackage{color}
\usepackage{xcolor}
\usepackage{graphics}
\usepackage{graphicx}
\usepackage{epstopdf}
\usepackage{amsmath,amsthm,amssymb,amsfonts}
\usepackage{mathrsfs}
\usepackage[garamont]{mathdesign}
\usepackage{enumitem}
\usepackage[norelsize,ruled,vlined,commentsnumbered]{algorithm2e}
\usepackage{multirow}
\usepackage[final,activate,verbose=true,auto=true]{microtype}
\usepackage{hyperref}
\hypersetup{colorlinks=true,linkcolor=blue,citecolor=blue,linktoc=page}
\usepackage{a4wide}

\newcommand{\eq}{\begin{equation}}
\newcommand{\qe}{\end{equation}}

\newcommand{\N}{\mathbb{N}}                %natural integers
\newcommand{\R}{\mathbb{R}}                     %real numbers
   
\newcommand{\nnd}{\mathbb{S}^{+}} 
\def\Cs{\mathscr{C}}
\def\Ms{\mathscr{M}}

\def\v{\mathbf{v}}
\def\x{x}
\def\K{\mathcal{X}}
\def\M{{M}}
\def\B{\mathbf{B}}
\def\C{\mathbf{C}}
\def\X{\Lambda}
\def\y{\mathbf{y}}

\newtheorem{thm}{Theorem}%[section]

\newtheorem{prop}[thm]{Proposition}
\newtheorem{defn}[thm]{Definition}
\newtheorem{rem}{Remark}

\date{\today}

\begin{document}
\sloppy

%Frontmatter
\begin{frontmatter}

\title{Approximate Optimal Designs for Multivariate Polynomial Regression}
\runtitle{Approximate Polynomial Optimal Designs}

\begin{aug}
\author{\fnms{Yohann} \snm{De Castro${}^{\star}$}\ead[label=e1]{yohann.decastro@math.u-psud.fr},}
\author{\fnms{Fabrice} \snm{Gamboa${}^{\circ}$}\ead[label=e2]{fabrice.gamboa@math.univ-toulouse.fr},}
\author{\fnms{Didier} \snm{Henrion${}^{\bullet}$}\ead[label=e3]{henrion@laas.fr},}
\author{\fnms{Roxana}~\snm{Hess${}^{\bullet}$}\ead[label=e4]{rhess@laas.fr}}
\and
\author{\fnms{Jean-Bernard} \snm{Lasserre${}^{\bullet}$}\ead[label=e5]{lasserre@laas.fr}.}
\affiliation{\small ${}^{\star}${\sl Laboratoire de Math\'ematiques d'Orsay}\\  Univ. Paris-Sud, CNRS,  Universit\'e Paris-Saclay, 91405 Orsay, France}
\affiliation{\small ${}^{\circ}${\sl Institut de Math\'ematiques de Toulouse}\\Univ. Paul Sabatier, CNRS, 118 route de Narbonne, 31062 Toulouse, France}
\affiliation{\small ${}^{\bullet}${\sl Laboratoire d'Analyse et d'Architecture des Systèmes}\\Université de Toulouse, CNRS, 7 avenue du col. Roche, 31400 Toulouse, France}
\runauthor{De Castro, Gamboa, Henrion, Hess and Lasserre}
\end{aug}

%%%%%%%%%%%%%%%%%%%%%%abstract%%%%%%%%%%%%%%%%%%%%%%%%%%%
\begin{abstract}
We introduce a new approach aiming at computing approximate optimal designs for multivariate polynomial regressions on compact (semi-algebraic) design spaces. We use the moment-sum-of-squares hierarchy of semidefinite programming problems to solve numerically the approximate optimal design problem. The geometry of the design is recovered via semidefinite programming duality theory. This article shows that the hierarchy converges to the approximate optimal design as the order of the hierarchy increases. Furthermore, we provide a dual certificate ensuring finite convergence of the hierarchy and showing that the approximate optimal design can be computed numerically with our method. As a byproduct, we revisit the equivalence theorem of the experimental design theory: it is linked to the Christoffel polynomial and it characterizes finite convergence of the moment-sum-of-square hierarchies.
\end{abstract}



\end{frontmatter}

%%%%%%%%%%%%%Title%%%%%%%%%%%%%%%%%%%%%%%%%%%%%%%%%%%%%%%
\maketitle 

%%%%%%%%%%%%%%%%%%%%%%%%%%%%%%%%%%%%%%%%%%%%%%%%%

\section{Introduction}

\subsection{Convex design theory}
The optimal experimental designs are computational and theoretical objects that aim at minimizing the uncertainty contained in the best linear unbiased estimators in regression problems. In this frame, the experimenter models the responses $z_1,\ldots,z_N$ of a random \textit{experiment} whose inputs are represented by a vector $t_{i}\in\R^{n}$ with respect to known \textit{regression functions} {${\mathbf f}_{1},\ldots,{\mathbf f}_{p}$}, namely
\[
z_{i}=\sum_{j=1}^{p}\theta_{j}{\mathbf f}_{j}(t_{i})+\varepsilon_{i}\,,\:\:i=1,\ldots,N,
\]
where $\theta_1,\ldots,\theta_p$ are unknown parameters that the experimenter wants to estimate, $\varepsilon_{i},\; i=1,\ldots,N$ are {i.i.d.\! centered square integrable random variables} and the inputs $t_{i}$~are chosen by the experimenter in a \textit{design space}~$\mathcal X\subseteq\R^{n}$. In this paper, we consider that the regression functions ${\mathbf F}$ are multivariate polynomials of degree at most~$d$.

Assume that the inputs {$t_i$}, for~$i=1,\ldots,N$, are chosen within a set of distinct points $x_1, \ldots, x_\ell$ with $\ell \leq N$, and let $n_k$ denote the number of times the particular point $x_k$ occurs among $t_1, \ldots, t_N$. This would be summarized by defining a design $\xi$ as follows
\eq
\label{eq:defDesignExact}
{\xi}:=\left(\begin{array}{ccc}x_{1} & \cdots & x_{\ell} \\ \frac{n_1}{N} & \cdots & \frac{n_{\ell}}{N}\end{array}\right),
\qe
whose first row gives distinct points in the design space $\mathcal X$ where the inputs parameters have to be taken and the second row indicates the experimenter which proportion of experiments (frequencies) have to be done at these points. We refer to the inspiring book of Dette and Studden~\cite{dette1997theory} and references therein for a complete overview on the subject of the theory of optimal design of experiments. We denote the \textit{information matrix} of $\xi$ by
\eq
\label{eq:defInformationMatrix}
{{\mathbf M}(\xi)}:=\sum_{i=1}^{\ell}w_{i}{\mathbf F}(x_{i})\,{\mathbf F}^{\top}\!(x_{i}),
\qe
where {${\mathbf F}:=({\mathbf f}_1,\ldots,{\mathbf f}_p)$} is the column vector of regression functions and $w_i:=n_i/N$ is the weight corresponding to the point $x_i$. In the following, we will not not distinguish between a design $\xi$ as in~\eqref{eq:defDesignExact} and a discrete probability measure on $\mathcal X$ with finite support given by the points $x_i$ and weights $w_i$.

Observe that the information matrix  belongs to $\nnd_{p}$, the space of symmetric nonnegative definite matrices of size $p$. For all $q\in[-\infty,1]$ define the function
\[
\phi_{q}\,:=\,\left\{\begin{array}{cll}\nnd_{p} & \to & \R \\ M & \mapsto & \phi_{q}(M)\end{array}\right.
\]
where for positive definite matrices $M$
\[
\phi_{q}(M):=\left\{\begin{array}{ll}(\frac1p\mathrm{trace}(M^{q}))^{1/q} & \mathrm{if}\ q\neq-\infty,0  \\ \det(M)^{1/p} & \mathrm{if}\ q=0  \\\lambda_{\min}(M) & \mathrm{if}\ q=-\infty\end{array}\right.
\]
and for nonnegative definite matrices $M$
\[
\phi_{q}(M):=\left\{\begin{array}{ll}(\frac1p\mathrm{trace}(M^{q}))^{1/q} & \mathrm{if}\ q\in(0,1]  \\ 0 & \mathrm{if}\ q\in[-\infty,0].\end{array}\right.
\]
We recall that $\mathrm{trace}(M)$, $\det(M)$ and $\lambda_{\mathrm{min}}(M)$ denote respectively the trace, determinant and least eigenvalue of the symmetric nonnegative definite matrix~$M$. These criteria are meant to be real valued, positively homogeneous, non constant, upper semi-continuous, isotonic (with respect to the Loewner ordering) and concave functions. 

Hence, an optimal design is a solution $\xi^{\star}$ to the following problem
\eq
\label{eq:defOptimumdesigns}
         \max \phi_q ({\mathbf M}(\xi))
\qe
where the maximum is taken over all $\xi$ of the form \eqref{eq:defDesignExact}. Standard criteria are given by the parameters $q=0,-1,-\infty$ and are referred to $D$, $A$ or $E$-optimum designs respectively. {As detailed in Section~\ref{sec:Relax_integer},  we restrict our attention to ‘‘{\it approximate}'' optimal designs where, by definition, we replace the set of ‘‘{\it feasible}'' matrices $\{{\mathbf M}(\xi) :  \xi \ \text{of the form \eqref{eq:defDesignExact}} \}$ by the larger set of all possible information matrices, namely the convex hull of $\{{\mathbf F}(x)\,{\mathbf F}^{\top}\!(x) : x\in\mathcal X\}$. To construct approximate optimal designs, we propose a two-step procedure presented in Algorithm~\ref{alg:GeneralAlgApproxDesign}. This procedure finds the information matrix ${\mathbf M}^\star$ of the approximate optimal design $\xi^\star$ and then it computes the support points $x^\star_i$ and the weights $w^\star_i$ of the design $\xi^\star$ in a second step.}

\subsection{Contribution}
{This paper introduces a general method to compute approximate optimal designs\textemdash in the sense of Kiefer's $\phi_{q}$-criteria\textemdash on a large variety of design spaces that we refer to as semi-algebraic sets, see~\cite{lasserre} or Section~\ref{sec:notation} for a definition. These can be understood as sets given by intersections and complements of 
superlevel sets of multivariate polynomials. An important distinguishing feature of the method is to not rely on any discretization of the design space which is in contrast to computational methods in previous works, {\it e.g.}, the algorithms described in \cite{torsney2009w,Sag15}.}

\begin{algorithm}[t]
  \SetAlgoLined
  \KwData{A compact semi-algebraic design space $\mathcal X$ defined as in \eqref{eq:defDesignSpaceSemiAegebraic}. 
    }
  \KwResult{An approximate optimal design $\xi$}
    \BlankLine
{		\begin{enumerate}
		\item Choose the two relaxation orders $\delta$ and $r$.
	    \item Solve the SDP relaxation \eqref{eq:MSDP} of order $\delta$ for a vector $\y_\delta^\star$.
		\item Either solve Nie's SDP relaxation \eqref{sdp-second} or the Christoffel polynomial SDP relaxation \eqref{sdp-three} of order $r$ for a vector $\y_r^\star$.
		\item If $\y_r^\star$ satisfies the rank condition (\ref{test}), then extract the optimal design $\xi$ from the truncated moment sequence as explained in Section~\ref{recoverMeasure}.
		\item Otherwise, choose larger values of $\delta$ and $r$ and go to Step 2.		
		\end{enumerate}}
\caption{Approximate Optimal Designs on Semi-Algebraic Sets}
\label{alg:GeneralAlgApproxDesign}
\end{algorithm}

We apply the moment-sum-of-squares hierarchy\textemdash referred to as the Lasserre hierarchy\textemdash of SDP problems to solve numerically and approximately the optimal design problem. {More precisely, we use an outer ‘‘{\it approximation}'' (in the SDP relaxation sense) of the set of moments of order $d$, see Section~\ref{sec:moment_cone} for more details. Note that these approximations are SDP representable so that they can be efficiently encoded numerically. Since the regressors are polynomials, the information matrix~$\mathbf M$ is a linear function of the moment matrix (of order $d$). Hence, our approach gives an outer approximation of the set of information matrices, which is SDP representable. As shown by the interesting works \cite{sagnol2013semidefinite, papp2012optimal}, the criterion $\phi_q$ is also SDP representable in the case where $q$ is rational. It proves that our procedure (depicted in Algorithm~\ref{alg:GeneralAlgApproxDesign}) makes use of two semidefinite programs and it can be efficiently used in practice. Note that similar two steps procedures have been presented in the literature, the reader may consult the interesting paper \cite{gaffke2014algorithms} which proposes a way of constructing approximate optimal designs on the hypercube.}

The theoretical guarantees are given by Theorem~\ref{th-sdp} (Equivalence theorem revisited for the finite order hierarchy) and Theorem~\ref{th3-asymptotics} (convergence of the hierarchy as the order increases). These theorems demonstrate the convergence of our procedure towards the approximate optimal designs as the order of the hierarchy increases. Furthermore, they give a characterization of finite order convergence of the hierarchy. In particular, our method recovers the optimal design when finite convergence of this hierarchy occurs. To recover the geometry of the design we use SDP duality theory and Christoffel polynomials involved in the optimality conditions. 

We have run several numerical experiments for which finite convergence holds leading to a surprisingly fast and reliable method to compute optimal designs. As illustrated by our examples, in polynomial regression model with degree order higher than one we obtain designs with points in the interior of the domain.

\subsection{Outline of the paper}
In Section \ref{sec:notation}, after introducing necessary notation, we shortly explain some basics on moments and moment matrices, and present the approximation of the moment cone via the Lasserre hierarchy. Section~\ref{sec:OptDesign} is dedicated to further describing optimal designs and their approximations. At the end of the section we propose a two step procedure to solve the approximate design problem, it is described in Algorithm~\ref{alg:GeneralAlgApproxDesign}. Solving the first step is subject to Section~\ref{idealProblem}. There, we find a sequence of moments~$\y^\star$ associated with the optimal design measure. Recovering this measure (step two of the procedure) is discussed in Section \ref{recoverMeasure}. We finish the paper with some illustrating examples and a short conclusion.

\pagebreak[3]

\section[Polynomial optimal designs]{Polynomial optimal designs and moments}
\label{sec:notation}
This section collects preliminary material on semi-algebraic sets, moments and moment matrices, using the notation of \cite{lasserre}. This material will be used to restrict our attention to \textit{polynomial} optimal design problems with polynomial regression functions and semi-algebraic design spaces.

\subsection{Polynomial optimal design}
Denote by $\R[x]$ the vector space of real polynomials in the variables $x=(x_1,\dotsc,x_n)$, and for $d\in\N$ define $\R[x]_d := \{p \in \R[x] : \deg{p} \leq d\}$ where $\deg p$ denotes the total degree of $p$.

We assume that the regression functions are multivariate polynomials, namely ${\mathbf F}=({\mathbf f}_1,\ldots,{\mathbf f}_p) \in (\R[x]_d)^p$. Moreover, we consider that the design space $\mathcal X\subset \R^n$ is a given closed basic semi-algebraic set
\eq
\label{eq:defDesignSpaceSemiAegebraic}
\mathcal X:= \{ x \in \R^m : g_j(x) \geqslant0, \:j=1,\ldots,m\}
\qe
for given polynomials $g_j \in {\R}[x]$, $j=1,\ldots,m$, whose degrees are denoted by $d_j$, $j=1,\ldots,m$. Assume that~${\mathcal X}$ is compact with an algebraic certificate of compactness. For example, one of the polynomial inequalities $g_j(x)\geqslant0$ should be of the form $R^2-\sum_{i=1}^n x_i^2 \geqslant0$ for a sufficiently large constant $R$.

Notice that those assumptions cover a large class of problems in optimal design theory, see for instance \cite[Chapter 5]{dette1997theory}. In particular, observe that the design space $\mathcal X$ defined by~\eqref{eq:defDesignSpaceSemiAegebraic} is not necessarily convex and note that the polynomial regressors $ {\mathbf F}$ can handle incomplete $m$-way $d$th degree polynomial regression.

The monomials $x_1^{\alpha_1}\cdots x_n^{\alpha_n}$, with $\alpha =(\alpha_1,\dotsc,\alpha_n) \in \N^n$, form a basis of the vector space $\R[x]$. We use the multi-index notation $x^\alpha:=x_1^{\alpha_1}\cdots x_n^{\alpha_n}$ to denote these monomials. In the same way, for a given $d\in \N$ the vector space~$\R[x]_d$ has dimension $s(d):=\binom{n+d}{n}$ with basis $(x^\alpha)_{|\alpha|\leq d}$, where $|\alpha| := \alpha_1+\cdots+\alpha_n$. We write
\begin{align*}
	&\v_d(x):=\\
	& (\underbrace{1}_{\text{degree 0}},\underbrace{x_1,\dotsc,x_n}_{\text{degree 1}},\underbrace{x_1^2,x_1x_2,\dotsc,x_1x_n,x_2^2,\dotsc,x_n^2}_{\text{degree 2}},\dotsc,\underbrace{\dotsc,x_1^d,\dotsc,x_n^d}_{\text{degree d}}\ \ )^\top
\end{align*}
for the column vector of the monomials ordered according to their degree, and where monomials of the same degree are ordered with respect to the lexicographic ordering. Note that, by linearity, there exists a unique matrix~$\mathfrak A$ of size ${p\times\binom{n+d}{n}}$ such that 
\eq
\label{eq:polynomial_basis}
\forall x\in\mathcal X,\quad{\mathbf F}(x)=\mathfrak A\,\v_d(x)\,.
\qe
The cone $\Ms_+({\mathcal X})$ of nonnegative Borel measures supported on $\mathcal X$ is understood as the dual to the cone of nonnegative elements of the space $\Cs({\mathcal X})$ of continuous functions on $\mathcal X$.

\pagebreak[3]

\subsection{Moments, the moment cone and the moment matrix}
\label{sec:moment_cone}
Given a positive measure $\mu \in \Ms_+(\mathcal{X})$ and $\alpha \in \N^n$, we call
\[
	y_\alpha = \int_\mathcal{X} x^\alpha d\mu
\]
the moment of order $\alpha$ of $\mu$. Accordingly, we call the sequence $\y = (y_\alpha)_{\alpha\in\N^n}$ the moment sequence of~$\mu$. Conversely, we say that $\y = (y_\alpha)_{\alpha\in\N^n}$ has a \textit{representing measure}, if there exists a measure~$\mu$ such that $\y$ is its moment sequence.

We denote by $\mathcal{M}_d(\mathcal{X})$ the convex cone of all truncated sequences $\y = (y_\alpha)_{|\alpha|\leq d}$ which have a representing measure supported on $\mathcal{X}$. We call it the \textit{moment cone} (of order $d$) of $\mathcal{X}$. It can be expressed as
\begin{align}
\label{eq:momcone}
	\mathcal{M}_d(\mathcal{X})
	\,:=\,\Big\{\y\in\R^{\binom{n+d}{n}}: &\exists\,\mu\in \Ms_+(\mathcal{X})\mbox{ s.t. }\\
	&y_\alpha=\int_\mathcal{X} x^\alpha\,d\mu,\:\:\forall \alpha \in \N^n, \:|\alpha|\leq d\Big\}.\,\notag
\end{align}
Let $\mathcal{P}_d(\mathcal{X})$ denotes the convex cone of all polynomials of degree at most $d$ that are nonnegative on $\mathcal X$. Note that we assimilate polynomials $p$ of degree at most~$d$ with a vector of dimension $s(d)$, which contains the coefficients of~$p$ in the chosen basis.

When $\mathcal X$ is a compact set, then
$\mathcal{M}_d(\mathcal{X})=\mathcal{P}_d(\mathcal{X})^\star$
and $\mathcal{P}_d(\mathcal{X})=\mathcal{M}_d(\mathcal{X})^\star$, see {\it e.g.,} \cite[Lemma 2.5]{malin15} or \cite{krein1977markov}.

When the design space is given by the univariate interval $\mathcal X=[a,b]$, {\it i.e.,} $n=1$, then this cone is representable using positive semidefinite Hankel matrices, which implies that convex optimization on this cone can be carried out with efficient interior point algorithms for \textit{semidefinite programming}, see {\it e.g.,}~\cite{maxdet}. Unfortunately, in the general case, there is no efficient representation of this cone. It has actually been shown in \cite{scheiderer} that the moment cone is not \textit{semidefinite representable}, {\it i.e.,} it cannot be expressed as the projection of a linear section of the cone of positive semidefinite matrices. However, we can use semidefinite approximations of this cone as discussed in Section \ref{sec:approxMomcon}.

Given a real valued sequence $\y = (y_\alpha)_{\alpha\in\N^n}$  we define the linear functional $L_\y:\R[x]\to\R$ which maps a polynomial $f=\sum_{\alpha\in\N^n} f_\alpha x^\alpha$ to
\[
	L_\y(f) = \sum_{\alpha \in \N^n} f_\alpha y_\alpha.
\]
A sequence $\y = (y_\alpha)_{\alpha\in\N^n}$ has a representing measure $\mu$ supported on $\mathcal{X}$ if and only if $L_\y(f)\geqslant0$ for all polynomials $f\in \R[x]$ nonnegative on $\mathcal{X}$ \cite[Theorem 3.1]{lasserre}.

The \textit{moment matrix} of a truncated sequence $\y = (y_\alpha)_{|\alpha|\leq 2d}$ is the $\binom{n+d}{n}\times\binom{n+d}{n}$-matrix $\M_d(\y)$ with rows and columns respectively indexed by integer $n$-tuples $\alpha \in \N^n,  |\alpha|,|\beta| \leq d$ and whose entries are given by
\[
	M_d(\y)(\alpha,\beta) = L_\y(x^\alpha x^\beta) = y_{\alpha+\beta}.
\]
It is symmetric ($M_d(\y)(\alpha,\beta)=M_d(\y)(\beta,\alpha)$), and linear in~$\y$. Further, if~$\y$ has a representing measure, then~$M_d(\y)$ is \textit{positive semidefinite} (written $M_d(\y) \succcurlyeq 0$).

Similarly, we define the \textit{localizing matrix} of a polynomial $f=\sum_{|\alpha|\leq r} f_\alpha x^\alpha \in \R[x]_r$ of degree~$r$ and a sequence $\y = (y_\alpha)_{|\alpha|\leq 2d+r}$ as the $\binom{n+d}{n}\times\binom{n+d}{n}$ matrix $M_d(f\y)$ with rows and columns respectively indexed by  $\alpha,\beta \in~\N^n,  |\alpha|,|\beta| \leq d$ and whose entries are given by
\[
	M_d(f\y)(\alpha,\beta) = L_\y(f(x)\,x^\alpha x^\beta) = \sum_{\gamma\in\N^n} f_\gamma y_{\gamma+\alpha+\beta}.
\]
If $\y$ has a representing measure $\mu$, then $M_d(f\y) \succcurlyeq 0$ for $f\in\R[x]_d$ whenever the support of $\mu$ is contained in the set $\{x\in\R^n: f(x)\geqslant0\}$.

Since $\mathcal{X}$ is basic semi-algebraic with a certificate of compactness, by Putinar's theorem\textemdash see for instance the book \cite[Theorem~3.8]{lasserre}, we also know the converse statement in the infinite case. Namely, it holds that  $\y = (y_\alpha)_{\alpha\in\N^n}$ has a representing measure $\mu \in\Ms_+(\mathcal{X})$ if and only if for all $d\in\N$ the matrices $M_d(\y)$ and~$M_d(g_j\y), \ j=1,\dots,m$, are positive semidefinite.

\subsection{Approximations of the moment cone}\label{sec:approxMomcon}
Letting $v_j:=\lceil d_j/2\rceil$, $j=1,\ldots,m$, denote half the degree of the~$g_j$, by Putinar's theorem, we can approximate the moment cone $\mathcal{M}_{2d}(\mathcal{X})$ by the following semidefinite representable cones for $\delta\in\N$:
\begin{align}
\label{eq:MSDP}
	\mathcal{M}_{2(d+\delta)}^{\mathsf{SDP}}(\mathcal{X}):= \Big\{&\y_{d,\delta}\in\R^{\binom{n+2d}{n}}\ :\ \exists \y_\delta\in\R^{\binom{n+2(d+\delta)}{n}}\mbox{ such that }\\\notag
&\y_{d,\delta}=(y_{\delta,\alpha})_{|\alpha|\leq 2d} \mbox{ and}\\\notag
&M_{d+\delta}(\y_\delta)\succcurlyeq 0,\ M_{d+\delta-v_j}(g_j\y_\delta)\succcurlyeq 0,\ j=1,\dotsc,m\Big\}.
\end{align}
By semidefinite representable we mean that the cones are projections of linear sections of semidefinite cones. Since $\mathcal{M}_{2d}(\mathcal{X})$ is contained in every $(\mathcal{M}_{2(d+\delta)}^{\mathsf{SDP}}(\mathcal{X}))_{\delta\in\mathbb N}$, they are outer approximations of the moment cone. Moreover, they form a nested sequence, so we can build the hierarchy
\eq
\label{eq:hierarchy}
	\mathcal{M}_{2d}(\mathcal{X})\subseteq\dots\subseteq\mathcal{M}_{2(d+2)}^{\mathsf{SDP}}(\mathcal{X})\subseteq\mathcal{M}_{2(d+1)}^{\mathsf{SDP}}(\mathcal{X})\subseteq\mathcal{M}_{2d}^{\mathsf{SDP}}(\mathcal{X}).
\qe
This hierarchy actually converges, meaning $\mathcal{M}_{2d}(\mathcal{X})=\overline{\bigcap_{\delta=0}^\infty \mathcal{M}_{2(d+\delta)}^{\mathsf{SDP}}(\mathcal{X})}$, where $\overline{A}$ denotes the topological closure of the set $A$.

Further, let $\Sigma[x]_{2d}\subseteq\R[x]_{2d}$ be the set of all polynomials that are sums of squares of polynomials (SOS) of degree at most $2d$, {\it i.e.,} $\Sigma[x]_{2d}=\{\sigma\in\R[x]: \sigma(x)=\sum_{i=1}^kh_i(x)^2 \text{ for some } h_i\in\R[x]_d \text{ and some }k\geq1\}$. The topological dual of~$\mathcal{M}_{2(d+\delta)}^{\mathsf{SDP}}(\mathcal{X})$ is a quadratic module, which we denote by $\mathcal P_{2(d+\delta)}^{\mathsf{SOS}}(\mathcal{X})$. It is given by
\begin{align}
\label{eq:SoS}
\mathcal P_{2(d+\delta)}^{\mathsf{SOS}}(\mathcal{X})&:=
\Big\{h=\sigma_0+\sum_{j=1}^m g_j\sigma_j: \mathrm{deg}(h)\leq2d,\,\\
& \sigma_0\in\Sigma[x]_{2(d+\delta)},\, \sigma_j\in\Sigma[x]_{2(d+\delta-\nu_j)},\ j=1,\dotsc,m\Big\}.\notag
\end{align}
Equivalently, see for instance \cite[Proposition 2.1]{lasserre}, $h\in\mathcal P_{2(d+\delta)}^{\mathsf{SOS}}(\mathcal{X})$ if and only if $h$ has degree less than $2d$ and there exist real symmetric and positive semidefinite matrices $Q_0$ and $Q_j,\ j=1,\dotsc,m$ of size $\binom{ n + d+\delta }{ n }\times\binom{ n + d+\delta }{ n }$ and $\binom{ n + d+\delta-\nu_j }{ n }\times\binom{ n + d+\delta-\nu_j }{ n }$ respectively, such that for any $x\in\R^n$
\begin{align*}
h(x)&=\sigma_0(x)+\sum_{j=1}^mg_j(x)\sigma_j(x)\\&=\v_{d+\delta}(x)^\top Q_0\v_{d+\delta}(x)+\sum_{j=1}^mg_j(x)\,\v_{d+\delta-\nu_j}(x)^\top Q_j\v_{d+\delta-\nu_j}(x)\,.
\end{align*}
The elements of $\mathcal P_{2(d+\delta)}^{\mathsf{SOS}}(\mathcal{X})$ are polynomials of degree at most $2d$ which are non-negative on~$\mathcal{X}$. Hence, it is a subset of $\mathcal{P}_{2d}(\mathcal{X})$.

\section{Approximate Optimal Design}
\label{sec:OptDesign}

\subsection{Problem reformulation in the multivariate polynomial case}
For all $i=1,\ldots,p$ and $x\in {\mathcal X}$, let ${\mathbf f}_{i}(x):=\sum_{|\alpha|\leq d}a_{i,\alpha}x^{\alpha}$ with appropriate $a_{i,\alpha}\in\R$ and note that $\mathfrak A=(a_{i,\alpha})$ where $\mathfrak A$ is defined by \eqref{eq:polynomial_basis}. For $\mu\in\Ms_+(\mathcal{X})$ with moment sequence $\y$ define the information matrix
\[
{\mathbf M}_d(\y):=\Big(\int_{{\mathcal X}}{\mathbf f}_{i}{\mathbf f}_{j}\mathrm d\mu\Big)_{1\leq i,j\leq p}=\Big(\sum_{|\alpha|,|\beta|\leq d}a_{i,\alpha}a_{j,\beta}y_{\alpha+\beta}\Big)_{1\leq i,j\leq p}=\sum_{|\gamma|\leq 2d} A_{\gamma}y_{\gamma},
\]
where we have set $A_{\gamma}:=\Big(\sum_{\alpha+\beta=\gamma}a_{i,\alpha}a_{j,\beta}\Big)_{1\leq i,j\leq p}$ for $|\gamma|\leq 2d$. Observe that it holds 
\eq
\label{eq:change-of-basis-in-information-matrix}
{\mathbf M}_d(\y)=\mathfrak AM_d(\y)\mathfrak A^{\top}.
\qe
If $\y$ is the moment sequence of $\mu=\sum_{i=1}^\ell w_i\delta_{x_i}$, where $\delta_x$ denotes the Dirac measure at the point $x\in\mathcal X$ and the $w_i$ are again the weights corresponding to the points $x_i$. Observe that ${\mathbf M}_d(\y)=\sum_{i=1}^{\ell}w_{i}{\mathbf F}(x_{i}){\mathbf F}^{\top}(x_{i})$ as in~\eqref{eq:defInformationMatrix}.

Consider the optimization problem
\begin{align}
\label{eq:Optimumdesigns_NLP}
	& \max\ \phi_q(M)\\
	& \text{s.t. } M=\sum_{|\gamma|\leq 2d}A_{\gamma}y_{\gamma} \succcurlyeq 0,\quad y_{\gamma} = \sum_{i=1}^\ell \frac{n_i}N{x^{\gamma}_i}, \quad \sum_{i=1}^{\ell} n_i = N,\notag\\
& \quad\quad x_i \in  {\mathcal X}, \:n_i \in {\mathbb N}, \:i=1,\ldots,\ell,\notag
\end{align}
where the maximization is with respect to $x_i$ and $n_i$, $i=1,\ldots,\ell$, subject to the constraint that the information matrix $M$ is positive semidefinite. By construction, it is equivalent to the original design problem~\eqref{eq:defOptimumdesigns}. In this form, Problem \eqref{eq:Optimumdesigns_NLP} is difficult because of the integrality constraints on the $n_i$ and the nonlinear relation between $\y$, $x_i$ and $n_i$. We will address these difficulties in the sequel by first relaxing the integrality constraints.

\subsection{Relaxing the integrality constraints}
\label{sec:Relax_integer}
In Problem \eqref{eq:Optimumdesigns_NLP}, the set of admissible frequencies $w_i=n_i/N$ is discrete, which makes it a potentially difficult combinatorial optimization problem. A popular solution is then to consider ``\textit{approximate}'' designs defined by 
\eq
\label{eq:defDesignApproximate}
\xi:=\left(\begin{array}{ccc}x_{1} & \cdots & x_{\ell} \\ w_{1} & \cdots & w_{\ell}\end{array}\right)\,,
\qe
where the frequencies $w_{i}$ belong to the unit simplex ${\mathcal W}:=\{w \in {\mathbb R}^\ell : 0 \leq w_i \leq 1, \: \sum_{i=1}^\ell w_i=1\}$. Accordingly, any solution to Problem~\eqref{eq:defOptimumdesigns}, where the maximum is taken over all matrices of type~\eqref{eq:defDesignApproximate}, is called ``\textit{approximate optimal design}'', yielding the following relaxation of Problem~\eqref{eq:Optimumdesigns_NLP}
\begin{align}
\label{eq:ApproxOptimumDesigns_NLP}
	& \max\ \phi_q(M)\\
	& \text{s.t. } M=\sum_{|\gamma|\leq 2d}A_{\gamma}y_{\gamma}\succcurlyeq 0,\quad y_{\gamma} = \sum_{i=1}^\ell w_i {x^{\gamma}_i},\notag\\
& \quad x_i \in  {\mathcal X}, \:w \in {\mathcal W},\notag
\end{align}
where the maximization is with respect to $x_i$ and $w_i$, $i=1,\ldots,\ell$, subject to the constraint that the information matrix $M$ is positive semidefinite. In this problem the nonlinear relation between $\y$, $x_i$ and~$w_i$ is still an issue.

\subsection{Moment formulation}
Let us introduce a two-step-procedure to solve the approximate optimal design Problem~\eqref{eq:ApproxOptimumDesigns_NLP}. For this, we first reformulate our problem again.

By Carath\'eodory's theorem, the 
subset of moment sequences in the truncated moment cone~$\mathcal M_{2d}(\mathcal{X})$ defined in \eqref{eq:momcone} 
and such that $y_0=1$, is exactly the set:
\begin{align*}
\Big\{\y\in\mathcal M_{2d}(\mathcal{X}): y_0=1\Big\}
=\Big\{\y\in\R^{\binom{n+2d}{n}}\ :&\  y_{\alpha} = \int_{\mathcal X} x^{\alpha}d\mu\quad \forall |\alpha|\leq 2d,\\&
 \: \mu=\sum_{i=1}^\ell w_i\delta_{x_i},\:x_i \in {\mathcal X},\: w\in {\mathcal W}\Big\},
\end{align*}
where $\ell\leq\binom{n+2d}{n}$, see the so-called Tchakaloff theorem \cite[Theorem B12]{lasserre}.
\pagebreak[3]

Hence, Problem \eqref{eq:ApproxOptimumDesigns_NLP} is equivalent to
\begin{align}
\label{eq:Step1_Moments}
	&\max\ \phi_q(M)\\
	& \text{s.t. } M=\sum_{|\gamma|\leq 2d} A_{\gamma} y_{\gamma}\succcurlyeq 0,\notag\\
		&\qquad \y\in\mathcal M_{2d}(\mathcal{X}),\ y_0=1,\notag
\end{align}
where the maximization is now with respect to the sequence $\y$. Moment problem \eqref{eq:Step1_Moments} is finite-dimensional and convex, yet the constraint $\y\in\mathcal M_{2d}(\mathcal{X})$ is difficult to handle. We will show that by approximating the truncated moment cone $\mathcal{M}_{2d}(\mathcal{X})$ by a nested sequence of semidefinite representable cones as indicated in \eqref{eq:hierarchy}, we obtain a hierarchy of finite dimensional semidefinite programming problems converging to the optimal solution of Problem~\eqref{eq:Step1_Moments}. Since semidefinite programming problems can be solved efficiently, we can compute a numerical solution to Problem \eqref{eq:ApproxOptimumDesigns_NLP}.

This describes step one of our procedure. The result of it is a sequence $\y^\star$ of moments. Consequently, in a second step, we need to find a representing atomic measure $\mu^\star$ of $\y^\star$ in order to identify the approximate optimal design~$\xi^\star$.

\section[Approximating the ideal problem]{The ideal problem on moments and its approximation}
\label{idealProblem}

For notational simplicity, let us use the standard monomial basis of $\R[x]_d$ for the regression functions, meaning ${\mathbf F} = ({\mathbf f}_1,\dotsc,{\mathbf f}_p):=(x^\alpha)_{|\alpha|\leq d}$ with $p=\binom{n+d}{n}$. This case corresponds to~$\mathfrak A=\mathrm{Id}$ in \eqref{eq:polynomial_basis}. Note that this is not a restriction, since one can get the results for other choices of ${\mathbf F}$ by simply performing a change of basis. Indeed, in view of \eqref{eq:change-of-basis-in-information-matrix}, one shall substitute $M_d(\y)$ by~$\mathfrak AM_d(\y)\mathfrak A^{\top}$ to get the statement of our results in whole generality; see Section~\ref{sec:generalBasis} for a statement of the results in this case. Different polynomial bases can be considered and, for instance, one may consult the standard framework described by the book \cite[Chapter 5.8]{dette1997theory}. 

For the sake of conciseness, we do not expose the notion of incomplete $q$-way $m$-th degree polynomial regression here but the reader may remark that the strategy developed in this paper can handle such a framework. 

Before stating the main results, we recall the gradients of the Kiefer's $\phi_q$ criteria in Table~\ref{tab:gradient}.

\begin{table}[!h]
\begin{tabular}{|c|c|c|c|c|}%c|}
  \hline
 Name & $D$-opt. & $A$-opt %& $T$-opt. 
 & $E$-opt. & generic case\\
  %\hline
q& $0$  & $-1$ %& $1$ 
& $-\infty$ & $q\neq0,-\infty$ \\
  \hline
 % &&&&&\\
$\phi_q(M)$ & $\det(M)^{\frac1p}$& $p(\mathrm{trace}(M^{-1}))^{-1} $ %& ${\mathrm{trace}(M)}/p$ 
& $\lambda_{\min}(M)$ & $\displaystyle\Big[\frac{\mathrm{trace}(M^{q})}p\Big]^{\frac1q} $\\
%& &&&&\\
%\hline
$\nabla\phi_q(M)$ & $\det(M)^{\frac1p}M^{-\frac1p}$ & $p(\mathrm{trace}(M^{-1})M)^{-2} $ %& $\mathrm{Id}/p$ 
& $\Pi_{\min}(M)$ & $\displaystyle\Big[\frac{\mathrm{trace}(M^{q})}p\Big]^{\frac1q-1} \frac{M^{q-1}}p$\\
%&&&&&\\
  \hline
\end{tabular}
\vspace*{0.3cm}
\caption{Gradients %and the equivalent criteria $F_q=-\log(\phi_q)-\log p/q$ 
of the Kiefer's $\phi_q$ criteria. We recall that $\Pi_{\min}(M)=uu^\top/||u||_2^2$ is defined only when the least eigenvalue of $M$ has multiplicity one and $u$ denotes a nonzero eigenvector associated to this least eigenvalue. If the least eigenvalue has multiplicity greater than $2$, then the sub gradient $\partial\phi_q(M)$ of  $\lambda_{\min}(M)$ is the set of all projectors on subspaces of the eigenspace associated to~$\lambda_{\min}(M)$, see for example \cite{Lew96a}. Notice further that $\phi_q$ is upper semi-continuous and is  a positively 
homogeneous function}
\label{tab:gradient}
\end{table}

\subsection{The ideal problem on moments}
The ideal formulation \eqref{eq:Step1_Moments} of our approximate optimal design problem reads
\begin{equation}
\label{sdp-ideal}
\begin{array}{rl}
	\rho=\ \displaystyle\max_{\y} &\phi_q( \M_d(\y))\\
		 \text{s.t.} &\y\in\mathcal M_{2d}(\mathcal{X}),\ y_0=1.
\end{array}\end{equation}
For this we have the following standard result.

\begin{thm}[Equivalence theorem]
\label{th-ideal}
Let $q\in(-\infty,1)$ and $\K\subseteq\R^n$ {be a compact semi-algebraic set as defined in~\eqref{eq:defDesignSpaceSemiAegebraic} and} with nonempty interior.
Problem~\eqref{sdp-ideal} is a convex optimization problem with a unique optimal solution 
$\y^\star\in \mathcal{M}_{2d}(\K)$. Denote by $p^\star_d$ the polynomial
\begin{equation}
%\notag
\label{christoffel-general}
\x\mapsto p_d^\star(\x):=\v_d(\x)^\top\M_d(\y^\star)^{q-1}\v_d(\x)=||\M_d(\y^\star)^{\frac{q-1}2}\v_d(\x)||_2^2.\end{equation}
Then $\y^\star$ is the vector of moments\textemdash up to order $2d$\textemdash of a discrete measure~$\mu^\star$ supported on at least $\binom{n+d}{n}$ and at most $\binom{n+2d}{n}$ points in the set $$\Omega:=\Big\{\x\in\K: \mathrm{trace}(\M_d(\y^\star)^{q})-p_d^\star(\x)=0\Big\},$$
In particular, the following statements are equivalent:
\begin{itemize}[label={$\circ$}]
\item $\y^\star\in \mathcal{M}_{2d}(\K)$ is the unique solution to Problem~\eqref{sdp-ideal};
\item $\y^\star\in {\Big\{\y\in\mathcal{M}_{2d}(\K):y_0=1\Big\}}$ and $p^\star{:=}\mathrm{trace}(\M_d(\y^\star)^{q})-p_d^\star\geqslant0$ on~$\K$.
\end{itemize}
\end{thm}

\pagebreak[3]

\begin{proof}
A general equivalence theorem for concave functionals of the information matrix is stated and proved in \cite[Theorem 1]{kiefer1974general}.
The case of $\phi_q$-criteria is tackled in \cite{pukelsheim2006optimal} and  \cite[Theorem~5.4.7]{dette1997theory}. In order to be self-contained and because the proof of our Theorem \ref{th-sdp} follows the same road map we recall a sketch of the proof in Appendix~\ref{proof:thm-ideal}.
\end{proof}

\begin{rem}[On the optimal dual polynomial]
\label{rem:opt poly}
The polynomial $p_d^\star$ contains all the information concerning the optimal design. Indeed, its level set $\Omega$ supports the optimal design points.
The polynomial is related to the so-called Christoffel function $($see Section \ref{sec:cricri}$)$. For this reason, in the sequel $p_d^\star$ in~\eqref{christoffel-general} will be called a Christoffel polynomial. Notice further that 
$$\mathcal{X}\subset\Big\{p_d^\star\leq \mathrm{trace}(\M_d(\y^\star)^{q})\Big\}.$$ 
Hence, the optimal design problem related to $\phi_q$ is similar to the standard problem of computational geometry consisting in minimizing the volume of a polynomial level set containing $\mathcal{X}$ (L\"owner-John's ellipsoid theorem). Here, the volume functional is replaced by $\phi_q(M)$ for the polynomial 
$||\M^{\frac{q-1}2}\v_d(\x)||_2^2$. We refer to 
$\cite{malin15}$ for a discussion and generalizations of L\"owner-John's ellipsoid theorem for general homogenous polynomials on non 
convex domains.
\end{rem}

\begin{rem}[Equivalence theorem for $E$-optimality]
Theorem~\ref{th-ideal} holds also for $q=-\infty$. This is the $E$-optimal design case, in which the objective function is not differentiable at points for which the least eigenvalue has multiplicity greater than $2$. We get that $\y^\star$ is the vector of moments\textemdash up to order $2d$\textemdash of a discrete measure~$\mu^\star$ supported on at most $\binom{n+2d}{n}$ points in the set 
$$\Omega:=\Big\{\x\in\K: \lambda_{\mathrm{min}}(\M_d(\y^\star))||u||_2^2-\Big(\sum_\alpha u_\alpha x^\alpha\Big)^2=0\Big\},$$
where $u=(u_\alpha)_{|\alpha|\leq2d}$ is a nonzero eigenvector of~$\M_d(\y^\star)$ associated to $\lambda_{\mathrm{min}}(\M_d(\y^\star))$. In particular, the following statements are equivalent
\begin{itemize}[label={$\circ$}]
\item $\y^\star\in \mathcal{M}_{2d}(\K)$ is a solution to Problem~\eqref{sdp-ideal};
\item $\y^\star\in{\{\y\in \mathcal{M}_{2d}(\K):y_0=1\}}$  and for all $x\in\K$, ${\Big(\sum_\alpha u_\alpha x^\alpha\Big)^2}\leq\lambda_{\mathrm{min}}(\M_d(\y^\star))||u||_2^2$.
\end{itemize}
Furthermore, if the least eigenvalue of $\M_d(\y^\star)$ has multiplicity one then $\y^\star\in \mathcal{M}_{2d}(\K)$ is unique. 
 \end{rem}

\subsection{Christoffel polynomials}
\label{sec:cricri}
In the case of $D$-optimality, it turns out that the unique optimal solution $\y^\star\in\mathcal{M}_{2d}(\mathcal{X})$ of Problem~\eqref{eq:Step1_Moments} can be characterized in terms of the {\it Christoffel polynomial} of degree $2d$ associated with 
an optimal measure~$\mu$ whose moments up to order $2d$ coincide with $\y^\star$. {Notice that in the paradigm of optimal design the Christoffel polynomial is the variance function of the multivariate polynomial regression model. Given a design, it is the variance of the predicted value of the model and so quantifies locally the uncertainty of the estimated response. 
We refer to~\cite{box1957multi} for its earlier introduction and the chapter \cite[Chapter 15]{pukelsheim2006optimal} for an overview of its properties and uses.}

\begin{defn}[Christoffel polynomial]
\label{def:christoffel}
Let $\y\in\R^{\binom{n+2d}{n}}$ be such that $\M_d(\y)\succ0$. Then there exists a family of orthonormal polynomials $(P_\alpha)_{\vert\alpha\vert\leq d}\subseteq\R[\x]_d$ satisfying
\[
L_\y(P_\alpha\,P_\beta)\,=\,\delta_{\alpha=\beta}\quad \mbox{and}\quad L_\y(\x^\alpha\,P_\beta)\,=\,0\quad\forall\alpha\prec\beta,
\]
where monomials are ordered with respect to the lexicographical ordering on~$\N^n$. We call the polynomial
\[
p_d:\ \x\mapsto p_d(x)\,:=\,\sum_{|\alpha|\leq d}P_\alpha(\x)^2,\quad x\in\R^n,
\]
the \textit{Christoffel polynomial} (of degree $d$) associated with $\y$.
\end{defn}

The Christoffel polynomial\footnote{Actually, what is referred to the {\it Christoffel function} in the literature is its reciprocal 
$x\mapsto1/p_d(x)$. {In optimal design, the Christoffel function is also called sensitivity function or information surface \cite{pukelsheim2006optimal}.}} can be expressed in different ways. For instance via the inverse of the moment matrix
by
\[
p_d(x) = \v_d(\x)^\top\M_d(\y)^{-1}\v_d(\x),\quad\forall x\in\R^n,
\]
or via its extremal property
\[\frac{1}{p_d(t)}\,=\,\min_{P\in\R[\x]_{d}}\Big\{\int P(x)^2\,d\mu(x)\::\: P(t)=1\,\Big\},\qquad \forall t\in\R^n,\]
when $\y$ has a representing measure $\mu$\textemdash when $\y$ does not have a representing measure
$\mu$ just replace $\int P(x)^2d\mu(x)$ with $L_\y(P^2)\,(=P^{\top}\M_d(\y)\,P$). For more details
the interested reader is referred to \cite{nips} and the references therein. Notice also that there is a regain of interest in the asymptotic study of the Christoffel function as it relies on eigenvalue marginal distributions of invariant random matrix ensembles, see for example~\cite{Led04}. 
\begin{rem}[Equivalence theorem for $D$-optimality]
%\label{}
In the case of $D$-optimal designs, observe that 
\[
t^\star:=\max_{\x\in\mathcal{X}}\ p^\star_d(\x)=\mathrm{trace}(\mathrm{Id})=\binom{n+d}{n}\,,
\]
where $p^\star_d$ given by \eqref{christoffel-general} for $q=0$. Furthermore, note that $p_d^\star$ is the Christoffel polynomial of degree $d$ of the $D$-optimal measure $\mu^\star$.

\end{rem}

\subsection{The SDP relaxation scheme}
Let $\K\subseteq\R^n$ be as defined in \eqref{eq:defDesignSpaceSemiAegebraic}, assumed to be compact.
So with no loss of generality (and possibly after scaling), assume that $\x\mapsto g_1(\x)=1-\Vert\x\Vert^2\geqslant0$ is one of the constraints defining~$\mathcal{X}$.

Since the ideal moment Problem \eqref{sdp-ideal} involves the moment cone  $\mathcal{M}_{2d}(\mathcal{X})$ which is not SDP representable, we use the hierarchy \eqref{eq:hierarchy} of outer approximations of the moment cone to relax Problem~\eqref{sdp-ideal} to an SDP problem. So for a fixed integer $\delta\geq1$ we consider the problem
\begin{equation}
\label{sdp}
\begin{array}{rl}
	\rho_\delta=\ \displaystyle\max_{\y} &\phi_q(\M_d(\y))\\
		 \text{s.t.} &\y\in\mathcal M_{2(d+\delta)}^{\mathsf{SDP}}(\mathcal{X}),\ y_0=1.
\end{array}\end{equation}
Since Problem~\eqref{sdp} is a relaxation of the ideal Problem \eqref{sdp-ideal}, necessarily $\rho_\delta\geq\rho$ for all $\delta$.
In analogy with Theorem \ref{th-ideal} we have the following result characterizing the solutions of the SDP relaxation \eqref{sdp} by means of Sum-of-Squares (SOS) polynomials. 

\begin{thm}[Equivalence theorem for SDP relaxations]
\label{th-sdp}
Let $q\in(-\infty,1)$ and {let $\K\subseteq\R^n$ be a compact semi-algebraic set as defined} in~\eqref{eq:defDesignSpaceSemiAegebraic} and be with non-empty interior. Then,
\begin{enumerate}[label={\alph*)}]
\item SDP Problem \eqref{sdp} has a {unique} optimal solution {$\y^\star\in \R^{\binom{n+2d}{n}}$.}
\item The moment matrix $\M_d({\y^\star})$ is positive definite. Let $p^\star_d$ be as defined in~\eqref{christoffel-general}, associated with ${\y^\star}$. Then {$p^\star:=\mathrm{trace}(\M_d(\y^\star)^{q})-p^\star_d$ is non-negative on $\mathcal{X}$ and $L_{\y^\star}(p^\star)=0$.}
\end{enumerate}

\pagebreak[3]

\noindent
In particular, the following statements are equivalent:
\begin{itemize}[label={$\circ$}]
\item %$\y^\star_\delta
${\y^\star}\in \mathcal M_{2(d+\delta)}^{\mathsf{SDP}}(\mathcal{X})$ is {the unique} solution to Problem~\eqref{sdp};% and $\y^\star_{d,\delta}$ is unique;
\item %$\y^\star_\delta\in \mathcal M_{2(d+\delta)}^{\mathsf{SDP}}(\mathcal{X})$ and $p^\star:x\mapsto\mathrm{trace}(\M_d(\y^\star_{d,\delta})^{q})-p_d^\star$
{$\y^\star\in \{\y\in\mathcal M_{2(d+\delta)}^{\mathsf{SDP}}(\mathcal{X}):y_0=1\}$ and $p^\star=\mathrm{trace}(\M_d(\y^\star)^{q})-p_d^\star$}$\in\mathcal P_{2(d+\delta)}^{\mathsf{SOS}}(\mathcal{X})$.
%\textcolor{red}{To be defined in section 2 for example}
\end{itemize}
\end{thm}

\pagebreak[3]

\begin{proof}
{We follow the same roadmap as in the proof of Theorem \ref{th-ideal}.}%We begin by proving Claims $a)$ and $b)$.
\begin{enumerate}[label={\alph*)}]
\item Let us prove that Problem \eqref{sdp} has an optimal solution. The feasible set is nonempty {with finite associated value}, since we can take as feasible point the vector $\tilde{\y}$ associated with the Lebesgue  measure on $\K$, scaled to be a probability measure. %It follows that $\phi_q(\M_d(\tilde{\y}))<\infty$. Hence, Slater's condition holds for Problem~\eqref{sdp} and $\rho_\delta<\infty$.
\par\medskip\noindent
Let $\y\in{\R^{\binom{n+2d}{n}}}$ be an arbitrary feasible solution and $\y_\delta\in\R^{\binom{n+2(d+\delta)}{n}}$ an arbitrary lifting 
of $\y$\textemdash recall the definition of $\mathcal M_{2(d+\delta)}^{\mathsf{SDP}}(\mathcal{X})$ given in~\eqref{eq:MSDP}. {Recall that $g_1(\x)=1-\Vert\x\Vert^2$. As $M_{d+\delta-1}(g_1\,y)\succeq0$ one deduces that $L_{y_\delta}(x_i^{2t}(1-\Vert x\Vert^2))\geq0$ for every $i=1,\ldots,n$, and all $t\leq d+\delta-1$. Expanding and using linearity of $L_y$ yields $1\geq \sum_{j=1}^n L_{y_\delta}(x_j^2)\geq L_{y_\delta}(x_i^2)$ for $t=0$ and $i=1,\ldots,n$. Next for $t=1$ and $i=1,\ldots,n$, 
\[0\leq L_{y_\delta}(x_i^2(1-\Vert x\Vert^2))\,=\,\underbrace{L_{y_\delta}(x_i^2)}_{\leq 1}-L_{y_\delta}(x_i^4)-\sum_{j\neq i}^n \underbrace{L_{y_\delta}(x_i^2x_j^2)}_{\geq0},\]
yields $L_{y_\delta}(x_i^4)\leq 1$. We may iterate this argumentation until we finally obtain $L_{y_\delta}(x_i^{2d+2\delta})\leq 1$, for all $i=1,\ldots,n$. Therefore by \cite[Lemma~4.3, page 110]{lass-netzer} (or \cite[Proposition~3.6, page 60]{lasserre}) one has
\begin{equation}
\label{bound}
\vert y_{\delta,\alpha}\vert \,\leq\,\max\Big\{\underbrace{y_{\delta,0}}_{=1},\ \max_i\{L_{\y_\delta}(x_i^{2(d+\delta)})\}\Big\}\,\leq\,1\,\ \forall|\alpha|\leq 2(d+\delta).
\end{equation}} 
This implies that the set of feasible liftings $\y_\delta$ is compact{, and therefore, the feasible set of~\eqref{sdp} is also compact. As the function $\phi_q$ is upper semi-continuous, the supremum in~\eqref{sdp} is attained at some optimal solution $\y^\star\in\R^{{s(2d)}}$.} It is unique due to {convexity of the feasible set and} strict concavity of the objective function $\phi_q$, {e.g.,} see \cite[Chapter~6.13]{pukelsheim2006optimal} for a proof.
\par\medskip

\item 
Let $\B_\alpha,\tilde{\B}_\alpha$ and $\C_{j\alpha}$ be real symmetric matrices such that
\begin{eqnarray*}
\sum_{|\alpha|\leq 2d} \B_\alpha \x^\alpha&=&\v_d(\x)\,\v_d(\x)^\top\\
\sum_{|\alpha|\leq 2(d+\delta)} \tilde{\B}_\alpha \x^\alpha&=&\v(\x)_{d+\delta}\,\v_{d+\delta}(\x)^\top\\
\sum_{|\alpha|\leq 2(d+\delta)} \C_{j\alpha} \x^\alpha&=&g_j(\x)\,\v_{d+\delta-v_j}(\x)\,\v_{d+\delta-v_j}(\x)^\top,\quad j=1,\ldots,m.
\end{eqnarray*}

Recall that it holds 
\[
\sum_{|\alpha|\leq 2d} \B_\alpha y_\alpha=M_d(\y)\,.
\]

{First, we notice that there exists a strictly feasible solution to \eqref{sdp} because the cone~$\mathcal M_{2(d+\delta)}^{\mathsf{SDP}}(\mathcal{X})$ has nonempty interior as a supercone of $\mathcal{M}_{2d}(\mathcal{X})$, which has nonempty interior by \cite[Lemma 2.6]{malin15}. Hence, Slater's condition\footnote{For the optimization problem $\max\,\{f(\x): A\x=\mathbf{b};\ \x\in \C\}$, where $A\in\R^{m\times n}$ and $\C\subseteq\R^n$ is a nonempty closed convex cone, Slater's condition holds, if there exists a feasible solution $\x$ in the interior of $\C$.}
 holds for \eqref{sdp}. Further, by an argument in \cite[Chapter 7.13]{pukelsheim2006optimal}) the matrix $\M_d(\y^\star)$ is non-singular. Therefore, $\phi_q$ is differentiable at $\y^\star$. Since additionally Slater's condition is fulfilled and $\phi_q$ is concave, this implies that the {\it Karush-Kuhn-Tucker} (KKT) optimality conditions\footnote{For the optimization problem $\max\,\{f(\x): A\x=\mathbf{b};\ \x\in \C\}$, where $f$ is differentiable, $A\in\R^{m\times n}$ and $\C\subseteq\R^n$ is a nonempty closed convex cone, the KKT-optimality conditions at a feasible point $\x$ state that there exist $\lambda^\star\in \R^m$ and $\mathbf{u}^\star\in \C^\star$ such that $A^\top\lambda^\star-\nabla f(\x)=\mathbf{u}^\star$ and $\langle \x,\mathbf{u}^\star\rangle=0$.}
at $\y^\star$ are necessary and sufficient for $\y^\star$ to be an optimal solution.

The KKT-optimality conditions at $\y^\star$ read
\[
\lambda^\star\,e_0-\nabla\phi_q(\M_d(\y^\star))\,=\,\hat{\mathbf{p}}^\star\quad\text{with }\hat{p}^\star(x):=\langle\hat{\mathbf{p}}^\star,\v_{2d}(x)\rangle\in\mathcal P_{2(d+\delta)}^{\mathsf{SOS}}(\mathcal{X}),
\]
where $\hat{\mathbf{p}}^\star\in\R^{s(2d)}$, $e_0=(1,0,\ldots,0)$, and $\lambda^\star$ is the dual variable associated with the constraint $y_0=1$. The complementarity condition reads $\langle \y^\star,\hat{\mathbf{p}}^\star\rangle=0$.

Recalling the definition~\eqref{eq:SoS} of the quadratic module $\mathcal P_{2(d+\delta)}^{\mathsf{SOS}}(\mathcal{X})$, we can express the membership $\hat{p}^\star(x)\in\mathcal P_{2(d+\delta)}^{\mathsf{SOS}}(\mathcal{X})$ more explicitly in terms of some {\it ``dual variables''} $\X_j\succcurlyeq0$, $j=0,\ldots,m$,}

\begin{equation}
\label{a1}
1_{\alpha=0} \, \lambda^\star-\langle\nabla\phi_q(\M_{d}(\y^\star)),\B_\alpha\rangle\,=\,\langle \X_0,\tilde{\B}_\alpha\rangle+\sum_{j=1}^m\langle \X_j,\C^j_\alpha\rangle,\quad|\alpha|\leq 2(d+\delta),
\end{equation}
{Then, for a lifting $\y_\delta^\star\in\R^{\binom{n+2(d+\delta)}{n}}$ of $\y^\star$ the complementary condition $\langle \y^\star,\hat{\mathbf{p}}^\star\rangle=0$ reads}
\begin{equation}
\label{a3}
\langle \M_{d+\delta}(\y^\star_\delta),\X_0\rangle =0;\quad \langle \M_{d+\delta-v_j}(\y^\star_\delta\,g_j),\X_j\rangle =0,\quad j=1,\ldots,m.
\end{equation}
Multiplying by $y^\star_{\delta,\alpha}$, summing up and using the complementarity conditions \eqref{a3} yields
\begin{equation}
\label{a4}
\lambda^\star-%\underbrace{
\langle \nabla\phi_q(\M_{d}(\y^\star%_{d,\delta}
)),\M_d(\y^\star%_{d,\delta}
)\rangle%}_{=\lambda^\star}
\,=\,\underbrace{\langle \X_0,\M_{d+\delta}(\y^\star_\delta)\rangle}_{=0}
+\sum_{j=1}^m \underbrace{\langle \X_j,\M_{d+\delta-v_j}(g_j\,\y^\star_\delta)\rangle}_{=0}.
\end{equation}
We deduce that 
\begin{equation}
\label{eq:lambdastar}
\lambda^\star=\langle \nabla\phi_q(\M_{d}(\y^\star_{d,\delta})),\M_d(\y^\star_{d,\delta})\rangle=\phi_q(\M_{d}(\y^\star_{d,\delta}))
\end{equation}
by the Euler formula for homogeneous functions.\par\medskip
\noindent
Similarly, multiplying by $\x^\alpha$ and summing up yields
\begin{align}
\notag
\lambda^\star&-\v_d(\x)^\top\nabla\phi_q(\M_{d}(\y^\star%_{d,\delta}
))\v_d(\x)\\
\notag&=
\Big\langle \X_0,\sum_{|\alpha|\leq 2(d+\delta)}\tilde{\B}_\alpha\,x^\alpha\Big\rangle 
+\sum_{j=1}^m\Big\langle \X_j,\sum_{|\alpha|\leq 2(d+\delta-v_j)}\C^j_\alpha\,x^\alpha\Big\rangle \\
\notag&=\underbrace{\Big\langle \X_0,\v(\x)_{d+\delta}\,\v_{d+\delta}(\x)^\top\Big\rangle}_{\sigma_0(x)} 
+\sum_{j=1}^mg_j(x)\,\underbrace{\Big\langle \X_j,\v_{d+\delta-v_j}(\x)\,\v_{d+\delta-v_j}(\x)^\top\Big\rangle}_{\sigma_j(x)} \\
&=\sigma_0(\x)+\sum_{j=1}^n\sigma_j(\x)\,g_j(\x)\notag\\
&{=\hat{p}^\star(x)\in\mathcal P_{2(d+\delta)}^{\mathsf{SOS}}(\mathcal{X}).}
\label{certificate}
\end{align}
{Note that $\sigma_0\in\Sigma[x]_{2(d+\delta)}$ and $\sigma_j\in\Sigma[x]_{2(d+\delta-d_j)}$, $j=1,\ldots,m$, by definition.}\par\medskip

{For $q\neq 0$ let $c^\star:=\binom{n+d}{n}\Big[{\binom{n+d}{n}}^{-1}{\mathrm{trace}(\M_d(\y^\star)^{q})}\Big]^{1-\frac1q}$. As $\M_d(\y^\star)$ is positive semidefinite and non-singular, we have $c^\star>0$. If $q=0$, let $c^\star:=1$ and replace $\phi_0( \M_d(\y^\star))$ by $\log \det \M_d(\y^\star)$, for which the gradient is~$\M_d(\y^\star)^{-1}$.
\par\medskip
Using Table~\ref{tab:gradient} we find that $c^\star\nabla\phi_q(\M_d(\y^\star))=\M_d(\y^\star)^{q-1}$. It follows that
\begin{align*}
c^\star\lambda^\star\overset{\eqref{eq:lambdastar}}{=}c^\star\langle\nabla\phi_q(\M_d(\y^\star)),\M_d(\y^\star)\rangle &=\mathrm{trace}(\M_d(\y^\star)^{q})\\
\text{and}\quad c^\star\langle\nabla\phi_q(\M_d(\y^\star)),\v_d(\x)\v_d(\x)^\top\rangle &\overset{\eqref{christoffel-general}}{=}p_d^\star(x)
\end{align*}

Therefore, Eq. \eqref{certificate} is equivalent to $p^\star:=c^\star\, \hat{p}^\star=c^\star\, \lambda^\star-p^\star_d\in\mathcal P_{2(d+\delta)}^{\mathsf{SOS}}(\mathcal{X})$. To summarize,
\begin{equation*}
	p^\star(x) = \mathrm{trace}(\M_d(\y^\star)^{q}) - p_d^\star(x)\in\mathcal P_{2(d+\delta)}^{\mathsf{SOS}}(\mathcal{X}).
\end{equation*}
We remark that all elements of $\mathcal P_{2(d+\delta)}^{\mathsf{SOS}}(\mathcal{X})$ are non-negative on $\mathcal{X}$ and that \eqref{a4} implies $L_{\y^\star}(p^\star)=0$. Hence, we have shown b).}
\end{enumerate}
The equivalence follows
{from the argumentation in b).}
\end{proof}

\begin{rem}
[Finite convergence]
If the optimal  solution $\y^\star%_{d,\delta}
$ of Problem~\eqref{sdp} is coming from a measure~$\mu^\star$ on~$\K$, that is $\y^\star%_{d,\delta}
\in \mathcal{M}_{2d}(\K)$, then $\rho_\delta=\rho$ and $\y^\star%_{d,\delta}
$ is the unique optimal solution of  Problem \eqref{sdp-ideal}. In addition, by {the proof of} Theorem~\ref{th-ideal}, $\mu^\star$ can be chosen to be atomic and supported on at least $\binom{n+d}{n}$ and at most $\binom{n+2d}{n}$ ``contact points'' on the level set $\Omega:=\{\x\in\K:\ \mathrm{trace}(\M_d(\y^\star)^{q})-p^\star_d(\x)=0\}$.
\end{rem}

 \begin{rem}[SDP relaxation for $E$-optimality]
Theorem~\ref{th-sdp} holds also for $q=-\infty$. This is the $E$-optimal design case, in which the objective function is not differentiable at points for which the least eigenvalue has multiplicity greater than $2$. We get that $\y^\star%_{d,\delta}
$ satisfies $\lambda_{\mathrm{min}}(\M_d(\y^\star%_{d,\delta}
))-\big(\sum_\alpha u_\alpha x^\alpha\big)^2\geqslant0$ for all $\x\in \K$ and $L_{\y^\star%_{d,\delta}
}(\big(\sum_\alpha u_\alpha x^\alpha\big)^2)=\lambda_{\mathrm{min}}(\M_d(\y^\star%_{d,\delta}
))$, where $u=(u_\alpha)_{|\alpha|\leq2d}$ is a nonzero eigenvector of~$\M_d(\y^\star%_{d,\delta}
)$ associated to $\lambda_{\mathrm{min}}(\M_d(\y^\star%_{d,\delta}
))$. 

In particular, the following statements are equivalent
\begin{itemize}[label={$\circ$}]
\item $\y^\star%_\delta
\in \mathcal M_{2(d+\delta)}^{\mathsf{SDP}}(\mathcal{X})$ is a solution to Problem~\eqref{sdp};
\item $\y^\star%_\delta
\in {\{\y\in\mathcal M_{2(d+\delta)}^{\mathsf{SDP}}(\mathcal{X}):y_0=1\}}$ and $p^\star(x)=\lambda_{\mathrm{min}}(\M_d(\y^\star))||u||_2^2-\Big(\sum_\alpha u_\alpha x^\alpha\Big)^2\in\mathcal P_{2(d+\delta)}^{\mathsf{SOS}}(\mathcal{X})$.
\end{itemize}
Furthermore, if the least eigenvalue of $\M_d(\y^\star%_{d,\delta}
)$ has multiplicity one then $\y^\star%_{d,\delta}
$ is unique. 
 \end{rem}

\subsection{Asymptotics}
We now analyze what happens when $\delta$ tends to infinity.
\begin{thm}
\label{th3-asymptotics}
Let $q\in(-\infty,1)$ {and $d\in\N$}. For every $\delta=0,1,2,\ldots,$ let $\y^\star_{{d,}\delta}$ be an optimal solution to \eqref{sdp} and $p^\star_{d,\delta}\in\R[\x]_{2d}$ the Christoffel polynomial associated with $\y^\star_{d,\delta}$ {defined} in Theorem \ref{th-sdp}.
Then,
\begin{enumerate}[label={\alph*)}]
\item $\rho_\delta\to\rho$ as $\delta\to\infty$, where $\rho$ is the supremum in \eqref{sdp-ideal}.
\item For every $\alpha\in\N^n$ with $|\alpha|\leq 2d$, {we have}
%\begin{equation*}
%\label{th3-1}
$\lim_{\delta\to\infty}y^\star_{{d,}\delta,\alpha}\,=\,y^\star_\alpha$, %\,=\,\int_\K \x^\alpha\,d\mu^\star,
%\end{equation*}
where $\y^\star=(y^\star_\alpha)_{|\alpha|\leq 2d}\in \mathcal{M}_{2d}(\K)$ is the unique optimal solution to \eqref{sdp-ideal}.
\item $p^\star_{d,\delta}\to p^\star_d$ as $\delta\to\infty$, where $p^\star_d$ is the Christoffel polynomial associated with $\y^\star$ defined in \eqref{christoffel-general}.
\item If the dual polynomial $p^\star:=\mathrm{trace}(\M_d(\y^\star)^{q})-p_d^\star$ {to Problem \eqref{sdp-ideal}}
%\sout{can be represented as a Sum-Of-Squares of order~$\delta$, i.e.,}
%$$ p^\star
belongs to $\mathcal P_{2(d+\delta)}^{\mathsf{SOS}}(\mathcal{X})$ %$
for some $\delta$, then finite convergence takes place, that is, $\y^\star_{d,\delta}$ is the unique optimal solution to Problem~\eqref{sdp-ideal} and $\y^\star_{d,\delta}$ has a representing measure, namely the target measure $\mu^\star$.
\end{enumerate}
\end{thm}
\begin{proof}
We prove the four claims consecutively.
\begin{enumerate}[label={\alph*)}]
\item For every $\delta$ complete the lifted finite sequence $\y^\star_\delta\in\R^{\binom{n+2(d+\delta)}{n}}$ with zeros to make it an infinite sequence $\y^\star_\delta=(y^\star_{\delta,\alpha})_{\alpha\in\N^n}$. Therefore, every such $\y^\star_\delta$
can be identified with an element of $\ell_\infty$, the Banach 
space of finite bounded sequences equipped with the supremum norm.
Moreover, Inequality~\eqref{bound} holds for every~$\y^\star_\delta$. Thus, denoting by $\mathcal{B}$ the unit ball of $\ell_\infty$ which is compact in the $\sigma(\ell_\infty,\ell_1)$ weak-$\star$ topology on $\ell_\infty$, we have $\y^\star_\delta\in \mathcal{B}$. By Banach-Alaoglu's theorem, there is an element $\hat{\y}\in\mathcal{B}$ and a converging subsequence $(\delta_k)_{k\in\N}$ such that
\begin{equation}
\label{conv}
\lim_{k\to\infty}y^\star_{\delta_k,\alpha}\,=\,\hat{y}_\alpha\qquad \forall \alpha\in\N^n.\end{equation}
Let $s\in \N$ be arbitrary, but fixed. By the convergence \eqref{conv} we also have
\begin{align*}
&\lim_{k\to\infty}\M_s(\y^\star_{\delta_k})\,=\,\M_s(\hat{\y})\succcurlyeq0;\\
&\lim_{k\to\infty}\M_s(g_j\,\y^\star_{\delta_k})\,=\,\M_s(g_j\,\hat{\y})\,\succcurlyeq0,\:j=1,\ldots,m.
\end{align*}
{Notice that the subvectors $\y^\star_{d,\delta}=(y_{\delta,\alpha}^\star)_{|\alpha|\leq 2d}$ with $\delta=0,1,2,\ldots$ belong to a compact set. Therefore, since $\phi_q(\M_d(\y^\star_{d,\delta}))<\infty$ for every $\delta$, we also have $\phi_q(\M_d(\hat{\y}))<\infty$. %\par\medskip\noindent
%Not sure, what this is good for as this follows directly from the following paragraph.
}

Next, by Putinar's theorem \cite[Theorem 3.8]{lasserre}, $\hat{\y}$ is the sequence of moments of some measure $\hat{\mu}\in \Ms_+(\K)$, and so $\hat{\y}_d=(\hat{y}_\alpha)_{|\alpha|\leq 2d}$ is a feasible solution to \eqref{sdp-ideal}, meaning $\rho\geq\phi_q(\M_d(\hat{\y}_d))$. On the other hand, as~\eqref{sdp} is a relaxation of \eqref{sdp-ideal}, we have $\rho\leq\rho_{\delta_k}$ for all $\delta_k$. So the convergence~\eqref{conv} yields
 \[\rho\leq\,\lim_{k\to\infty}\rho_{\delta_k}\,=\,\phi_q( \M_d(\hat{\y}_d)),\]
 which proves that $\hat{\y}$ is an optimal solution to \eqref{sdp-ideal}, and $\lim_{\delta\to\infty}\rho_\delta=\rho$.

\item As the optimal solution to \eqref{sdp-ideal} is unique, we have $\y^\star=\hat{\y}_d$ {with $\hat{\y}_d$ defined in the proof of a)} and the whole sequence $(\y^\star_{d,\delta})_{\delta\in\N}$ converges to $\y^\star$, that is, {for $\alpha\in\N^n$ with $|\alpha|\leq 2d$ fixed}
\begin{equation}
\label{conv2}
{\lim_{d,\delta\to\infty}y^\star_{\delta,\alpha}\,=}\,\lim_{\delta\to\infty}y^\star_{\delta,\alpha}\,=\,\hat{y}_\alpha=y^\star_\alpha.%\qquad \forall |\alpha|\leq 2d.
\end{equation}

\item %To show (c) i
It suffices to observe that the coefficients of Christoffel polynomial $p^\star_{d,\delta}$ are continuous functions of the
moments ${(y^\star_{d,\delta,\alpha})_{|\alpha|\leq 2d}=}(y^\star_{\delta,\alpha})_{|\alpha|\leq 2d}$. Therefore, by the convergence \eqref{conv2} 
one has $p^\star_{d,\delta}\to p^\star_d$ where $p^\star_d\in\R[\x]_{2d}$ as in Theorem \ref{th-ideal}.
\end{enumerate}
The last point follows directly observing that, in this case, the two Programs~{\eqref{sdp-ideal} and \eqref{sdp}} satisfy the same KKT conditions.
\end{proof}

\subsection{General regression polynomial bases}
\label{sec:generalBasis}

We return to the general case described by a matrix $\mathfrak A$ of size ${p\times\binom{n+d}{n}}$ such that the regression polynomials satisfy ${\mathbf F}(x)=\mathfrak A\,\v_d(x)$ for all $x\in\mathcal X$. {Without loss of generality, we can assume that the rank of $\mathfrak A$ is $p$, {\it i.e.,} the regressors ${\mathbf f}_{1},\ldots,{\mathbf f}_{p}$ are linearly independent.} 
Now, the objective function becomes~$\phi_q(\mathfrak AM_d(\y)\mathfrak A^{\top})$ at point~$\y$. Note that the constraints on~$\y$ are unchanged,~i.e., 
\begin{itemize}
\item
$\y\in\mathcal M_{2d}(\mathcal{X}),\ y_0=1$ in the ideal problem,
\item 
$\y\in\mathcal M_{2(d+\delta)}^{\mathsf{SDP}}(\mathcal{X}),\ y_0=1$ in the SDP relaxation scheme.
\end{itemize}
We recall the notation ${\mathbf M}_d(\y):=\mathfrak AM_d(\y)\mathfrak A^{\top}$ and we get that the KKT conditions are given by
\[\forall \x\in\K,\quad\phi_q({\mathbf M}_d(\y))-\underbrace{{\mathbf F}(\x)^\top\nabla\phi_q({\mathbf M}_d(\y))\,{\mathbf F}(\x)}_{\text{proportional to }p_d^\star(\x)}\,=\,p^\star(\x)\]
where 
\begin{itemize}
\item
$p^\star\in\mathcal{M}_{2d}(\mathcal{X})^\star\:(=\mathcal{P}_{2d}(\mathcal{X}))$ in the ideal problem,
\item
$p^\star\in\mathcal M_{2(d+\delta)}^{\mathsf{SDP}}(\mathcal{X})^\star\:(=\mathcal P_{2(d+\delta)}^{\mathsf{SOS}}(\mathcal{X}))$ in the SDP relaxation scheme.
\end{itemize}
Our analysis leads to the following equivalence results in this case.

\begin{prop}
Let $q\in(-\infty,1)$ and let $\K\subseteq\R^n$ be a compact {semi-algebraic set as defined} in~\eqref{eq:defDesignSpaceSemiAegebraic} and with nonempty interior.
Problem~\eqref{eq:ApproxOptimumDesigns_NLP} is a convex optimization problem with an optimal solution 
$\y^\star\in \mathcal{M}_{2d}(\K)$. Denote by $p^\star_d$ the polynomial
\begin{equation}
%\notag
\label{christoffel-general-2}
\x\mapsto p_d^\star(\x):={\mathbf F}(\x)^\top{\mathbf M}_d(\y)^{q-1}{\mathbf F}(\x)=||{\mathbf M}_d(\y)^{\frac{q-1}2}{\mathbf F}(\x)||_2^2.\end{equation}
Then $\y^\star$ is the vector of moments\textemdash up to order $2d$\textemdash of a discrete measure~$\mu^\star$ supported on at least ${p}$ points and at most {$\overline s$} points where
{
\[
\overline s\leq \min\bigg[1+\frac{p(p+1)}2, \binom{n+2d}{n}\bigg]
\] }
$($see Remark~\ref{rem:chachacha}\,$)$ in the set $\Omega:=\{\x\in\K: \mathrm{trace}({\mathbf M}_d(\y)^{q})-p_d^\star(\x)=0\}$.

\pagebreak[3]

\noindent
In particular, the following statements are equivalent:
\begin{itemize}[label={$\circ$}]
\item $\y^\star\in \mathcal{M}_{2d}(\K)$ is the solution to Problem~\eqref{sdp-ideal};
\item %$\y^\star\in \mathcal{M}_{2d}(\K)$  and for all $x\in\K$, ${\mathbf F}(\x)^T{\mathbf M}_d(\y)^{q-1}{\mathbf F}(\x)\leq\mathrm{trace}({\mathbf M}_d(\y)^{q})$.
{$\y^\star\in \{\y\in\mathcal{M}_{2d}(\mathcal{X}):y_0=1\}$ and $p^\star:=\mathrm{trace}({\mathbf M}_d(\y)^{q})-p_d^\star(x)\geqslant 0$ on~$\mathcal{X}$.}
\end{itemize}
Furthermore, if $\mathfrak A$ has full column rank then $\y^\star$ is unique.
\end{prop}

The SDP relaxation is given by the program
\begin{equation}
\label{sdp-2}
\begin{array}{rl}
	\rho_\delta=\ \displaystyle\max_{\y} &\phi_q({\mathbf M}_d(\y))\\
		 \text{s.t.} &\y\in\mathcal M_{2(d+\delta)}^{\mathsf{SDP}}(\mathcal{X}),\ y_0=1,
%\rho=\displaystyle\sup_\y \:\{\:\log {\rm det} \M_d(\y) :\: y_0=1; \quad \y\in M_d(\K)\,\}
\end{array}
\end{equation}
for which it is possible to prove the following result.

\begin{prop}
Let $q\in(-\infty,1)$ and let $\K\subseteq\R^n$ be a compact {semi-algebraic set as defined} in~\eqref{eq:defDesignSpaceSemiAegebraic} and with nonempty interior. Then,
\begin{enumerate}[label={\alph*)}]
\item %If $\mathfrak A$ has full column rank, then SDP Problem \eqref{sdp-2} has a unique optimal solution $\y^\star_{d,\delta}\in \R^{\binom{n+2d}{n}}$. 
{SDP Problem \eqref{sdp-2} has an optimal solution $\y^\star_{d,\delta}\in \R^{\binom{n+2d}{n}}$.}
\item Let $p^\star_d$ be as defined in~\eqref{christoffel-general-2}, associated with $\y^\star%_{d,\delta}
$. Then
$p^\star:=\mathrm{trace}({\mathbf M}_d(\y^\star_{d,\delta})^{q})-p^\star_d(\x)\geqslant0$ on $\K$ and $L_{\y^\star_{d,\delta}}(p^\star)=0$.
\end{enumerate}
In particular, the following statements are equivalent:
\begin{itemize}[label={$\circ$}]
\item $\y^\star%_\delta
\in \mathcal M_{2(d+\delta)}^{\mathsf{SDP}}(\mathcal{X})$ is a solution to Problem~\eqref{sdp};
\item $\y^\star%_\delta
\in {\{\y\in\mathcal M_{2(d+\delta)}^{\mathsf{SDP}}(\mathcal{X}):y_0=1\}}$ and %$x\mapsto\mathrm{trace}({\mathbf M}_d(\y^\star_{d,\delta})^{q})-{\mathbf F}(\x)^\top{\mathbf M}_d(\y^\star_{d,\delta})^{q-1}{\mathbf F}(\x)$ belongs to 
{$p^\star=\mathrm{trace}({\mathbf M}_d(\y^\star)^{q})-p_d^\star\in$}$\mathcal P_{2(d+\delta)}^{\mathsf{SOS}}(\mathcal{X})$.
\end{itemize}
Furthermore, if $\mathfrak A$ has full column rank then $\y^\star%_{d,\delta}
$ is unique.
\end{prop}

\section{Recovering the measure}
\label{recoverMeasure}

By solving step one as explained in Section \ref{idealProblem}, we obtain a solution $\y^\star%_d
$
of SDP Problem \eqref{sdp}. %However, we do not know, if $\y^\star_d$ comes from a measure. This would be the case, if we can find an atomic measure having these moments and yielding the same value in Problem \eqref{sdp}. For this,
{As $\y^\star{\in}\mathcal M_{2(d+\delta)}^{\mathsf{SDP}}(\mathcal{X})$, it is likely that it comes from a measure. If this is the case, by Tchakaloff's theorem, there exists an atomic measure supported on at most $s(2d)$ points having these moments. For computing the atomic measure,}
we propose two approaches: A first one which follows a procedure by Nie \cite{nie}, and a second one which uses properties of the Christoffel polynomial associated with~$\y^\star$.%_d$.

These approaches have the benefit that they can numerically certify finite convergence of the hierarchy.

\subsection{Via Nie's method}
This approach to recover a measure from its moments is based on a formulation proposed by Nie in~\cite{nie}. 

Let $\y^\star=(y^\star_{\alpha})_{|\alpha|\leq 2d}$ %be a solution to \eqref{sdp}. 
{a finite sequence of moments.} For $r\in\N$ consider the SDP problem
\begin{equation}
\label{sdp-second}
\begin{array}{rl} \displaystyle\min_{\y_r} & L_{\y_r}(f_r)\\
\mbox{s.t.}& \M_{d+r}(\y_r)\,\succcurlyeq\,0,\\
& \M_{d+r-v_j}(g_j\,\y_r)\,\succcurlyeq\,0,\quad j=1,\ldots,m,\\
& y_{r,\alpha}=y^\star_{%\delta,
\alpha},\quad \forall \alpha\in\N^n,\ |\alpha|\leq 2d,
\end{array}\end{equation}
where $\y_r\in\R^{\binom{n+2(d+r)}{n}}$ and $f_r\in\R[\x]_{2(d+r)}$ is a randomly generated polynomial strictly positive on $\K$, and again $v_j=\lceil d_j/2\rceil$, $j=1,\ldots,m$. 
We check whether the optimal solution $\y^\star_r$ of \eqref{sdp-second} satisfies the rank condition
\begin{equation}
\label{test}
{\rm rank}\:\M_{d+r}(\y_r^\star)\,=\,{\rm rank}\:\M_{d+r-v}(\y_r^\star),\end{equation}
where $v:=\max_j v_j$. Indeed if \eqref{test} holds then $\y_r^\star$ is the sequence of moments (up to order~$2r$) of a measure supported on $\mathcal{X}$; see \cite[Theorem 3.11, p. 66]{lasserre}.
If the test is passed, then we stop, otherwise {we increase $r$ by one and repeat the procedure.} %we repeat with $r:=r+1$. 
{As $\y^\star\in \mathcal{M}_{2d}(\mathcal{X})$, the rank condition~\eqref{test} is satisfied for a sufficiently large value of $r$.}

{We extract the support points $\x_1,\dotsc,\x_\ell\in\mathcal{X}$ of the representing atomic measure of $\y^\star_r$, and~$\y^\star$ respectively, as described in \cite[Section 4.3]{lasserre}.}

Experience reveals that in most cases it is enough to use the following polynomial
\[\x\mapsto f_r(\x)=\sum_{|\alpha|\leq d+r}\x^{2\alpha}=||\v_{d+r}(\x)||_2^2\]
instead of using a random positive polynomial on $\K$. In Problem~\eqref{sdp-second} this corresponds to minimizing the trace of $\M_{d+r}(\y)$\textemdash and so induces an optimal solution $\y$ with low rank matrix~$\M_{d+r}(\y)$.

\subsection{Via the Christoffel polynomial}
Another possibility to recover the atomic representing measure of $\y%_d
^\star$ is to find the zeros of the polynomial $p^\star(x)=\mathrm{trace}(\M_d(\y^\star%_{d,\delta}
)^{q})-p^\star_d(x)$, where $p^\star_d$ is the Christoffel polynomial associated with $\y%_d
^\star$ %on $\K$ as 
defined in \eqref{christoffel-general}, that is, $p^\star_d(x)=\v_d(\x)^\top\M_d(\y^\star%_{d,\delta}
)^{q-1}\v_d(\x)$. In other words, we compute the set $\Omega=\{x\in\mathcal{X}:\mathrm{trace}(\M_d(\y^\star%_{d,\delta}
)^{q})-p_d^\star(x)=0\}$, which due to Theorem \ref{th-sdp} is the support of the atomic representing measure.

{To that end we minimize $p^\star$ on $\mathcal{X}$. As the polynomial $p^\star$ is non-negative on $\mathcal{X}$, the minimizers are exactly $\Omega$. For minimizing $p^\star$, we use the Lasserre hierarchy of lower bounds, that is, we solve the semidefinite program}
\begin{equation}
\label{sdp-three}
\begin{array}{rl} \displaystyle\min_{\y_r}&L_{\y_r}(p^\star)\\
\mbox{s.t.}& \M_{d+r}(\y_r)\,\succcurlyeq\,0,\quad y_{r,0}=1,\\
& \M_{d+r-v_j}(g_j\,\y_r)\,\succcurlyeq\,0,\quad j=1,\ldots,m,
\end{array}
\end{equation}
where $\y_r\in\R^{\binom{n+2(d+r)}{n}}$. %

Since $p^\star_d$ is associated with the optimal solution to \eqref{sdp} for some given $\delta\in\N$, by Theorem \ref{th-sdp}, it satisfies the Putinar certificate \eqref{certificate} of positivity on~$\K$. Thus, the value of Problem \eqref{sdp-three} is zero for all $r\geqslant\delta$. Therefore, for every feasible solution $\y_r$ of \eqref{sdp-three} one has $L_{\y_r}(p^\star)\geq0$ (and $L_{\y^\star_d}(p^\star)=0$ for~$\y^\star_d$ an optimal solution of \eqref{sdp}).

When condition \eqref{test} is fulfilled, the optimal solution $\y^\star_r$ comes from a measure. We extract the support points $\x_1,\dotsc,\x_\ell\in\mathcal{X}$ of the representing atomic measure of $\y^\star_r$, and $\y^\star$ respectively, as described in \cite[Section 4.3]{lasserre}.

Alternatively, we can solve the SDP
\begin{equation}
\label{sdp-four}
\begin{array}{rl} \displaystyle\min_{\y_r} &{\rm trace}(\M_{d+r}({\y_r}))\\
\mbox{s.t.}& L_{\y_r}(p^\star)\,=\,0,\\
&\M_{d+r}(\y_r)\,\succcurlyeq\,0,\quad y_{r,0}=1,\\
& \M_{d+r-v_j}(g_j\,\y_r)\,\succcurlyeq\,0,\quad j=1,\ldots,m,
\end{array}
\end{equation}
where $\y_r\in\R^{\binom{n+2(d+r)}{n}}$. {This problem also searches for a moment sequence of a measure supported on the zero level set of $p^\star$. Again, if condition \eqref{test} is holds, the finite support can be extracted.}

\subsection{Calculating the corresponding weights}
\label{weights}
After recovering the support $\{x_1,\dotsc,x_\ell\}$ of the atomic representing measure by one of the previously presented methods, we might be interested in also computing the corresponding weights $\omega_1,\dotsc,\omega_\ell$. These can be calculated easily by solving the following linear system of equations: $\sum_{i=1}^\ell \omega_i x_i^\alpha = y_{%d,
\alpha}^\star$ for all $|\alpha|\leq {2}d$, {\it i.e.,} $\int_{\mathcal{X}} x^\alpha\mu^\star(dx) = y^\star_{%d,
\alpha}$.

\section{Examples}
We illustrate the procedure on {six} examples: a univariate one, {four examples} in the plane and one example on the three-dimensional sphere. We concentrate on $D$-optimal designs, namely $q=0$.% and $T$-optimal designs ($q=1$).

All examples are modeled by GloptiPoly 3 \cite{gloptipoly} and YALMIP \cite{yalmip} and solved by MOSEK~7~\cite{mosek} or SeDuMi under the MATLAB R2014a environment. We ran the experiments on an HP EliteBook %840 G1
with 16-GB RAM memory and an Intel Core i5-4300U processor. % under a Windows 7 Professional 64-bit operating system.
We do not report computation times, since they are negligible for our small examples.

\subsection{Univariate unit interval}\label{expl1}
We consider as design space the interval $\mathcal{X}=[-1,1]$ and on it the polynomial measurements $\sum_{j=0}^d \theta_jx^j$ with unknown parameters $\theta\in\R^{d+1}$. 
%\subsubsection{D-optimal design}
%\label{expl1}
To compute the $D$-optimal design we first solve Problem~\eqref{sdp}, in other words
\begin{equation}\label{expl1D}
\begin{array}{rl}
	\displaystyle\max_{\y_\delta} &\log \det \M_d(\y_\delta)\\
	\text{s.t.}& \M_{d+\delta}(\y_\delta)\,\succcurlyeq\,0,\\
	&\M_{d+\delta-1}((1-\Vert x\Vert^2)\, \y_\delta)\,\succcurlyeq\,0,\\
	&y_{\delta,0}=1
\end{array}\end{equation}
for {$\y_\delta\in\R^{s(2(d+\delta))}$ and} given regression order $d$ and relaxation order $d+\delta${, and then taking the truncation $\y^\star:=(y^\star_{\delta,\alpha})_{|\alpha|\leqslant 2d}$ of an optimal solution $\y^\star_\delta$}. For instance, for $d=5$ and $\delta=0$ we obtain the sequence $\y^\star\approx\,$(1, 0, 0.56, 0, 0.45, 0, 0.40, 0, 0.37, 0, $0.36)^\top$.

Then, to recover the corresponding atomic measure from the sequence $\y^\star$ we solve the problem
\begin{equation}\label{expl1R}
\begin{array}{rl} \displaystyle\min_\y & {\rm trace}\:\M_{d+r}({\y_r})\\
\mbox{s.t.}& \M_{d+r}({\y_r})\,\succcurlyeq\,0\\
& \M_{d+r-1}{((1-x^2)\y_r)}\,\succcurlyeq\,0,\\
& y_{\alpha}={y^\star_{r,\alpha}},\quad |\alpha|\leq 2d,
\end{array}\end{equation}
and find the points -1, -0.765, -0.285, 0.285, 0.765 and 1 (for $d=5$, $\delta$=0, $r=1$). As a result, our optimal design is the weighted sum of the Dirac measures supported on these points. The points match with the known analytic solution to the problem,  which are the critical points of the Legendre polynomial, see {\it e.g.,} \cite[Theorem 5.5.3, p.162]{dette1997theory}. {In this case, we know explicitly the optimal design, its support is located at the roots of the polynomial $t \to (1-t^2)P'_d(t)$ where $P'_d$ denotes  the derivative of  the Legendre polynomial of degree $d$, and its weights are all equal to $1/(1+d)$. Now, observe that the roots of $p^\star$ have degree $2$ in the interior of $[-1,1]$ (there are $d-1$ roots corresponding exactly to the roots of $P'_d$) and degree $1$ on the edges (corresponding exactly to the roots of $(1-t^2)$). Observe also that $p^\star$ has degree $2d$. We deduce that $p^\star$ equals $t \to (1-t^2)(P'_d(t))^2$ up to a multiplicative constant.} Calculating the corresponding weights as described in Section~\ref{weights}, we find $\omega_1=\dotsb=\omega_6\approx 0.166$ as prescribed by the theory.

Alternatively, we compute the roots of the polynomial
$x\mapsto p^\star(x)=6-p^\star_5(x)$, where $p^\star_5$ is the Christoffel polynomial of degree $2d=10$ on $\mathcal{X}$ and find the same points as in the previous approach by solving Problem~\eqref{sdp-four}. See Figure \ref{fig:Ch1} for the graph of the Christoffel polynomial of degree 10.

\begin{figure}[h!]
\includegraphics[scale=.8]{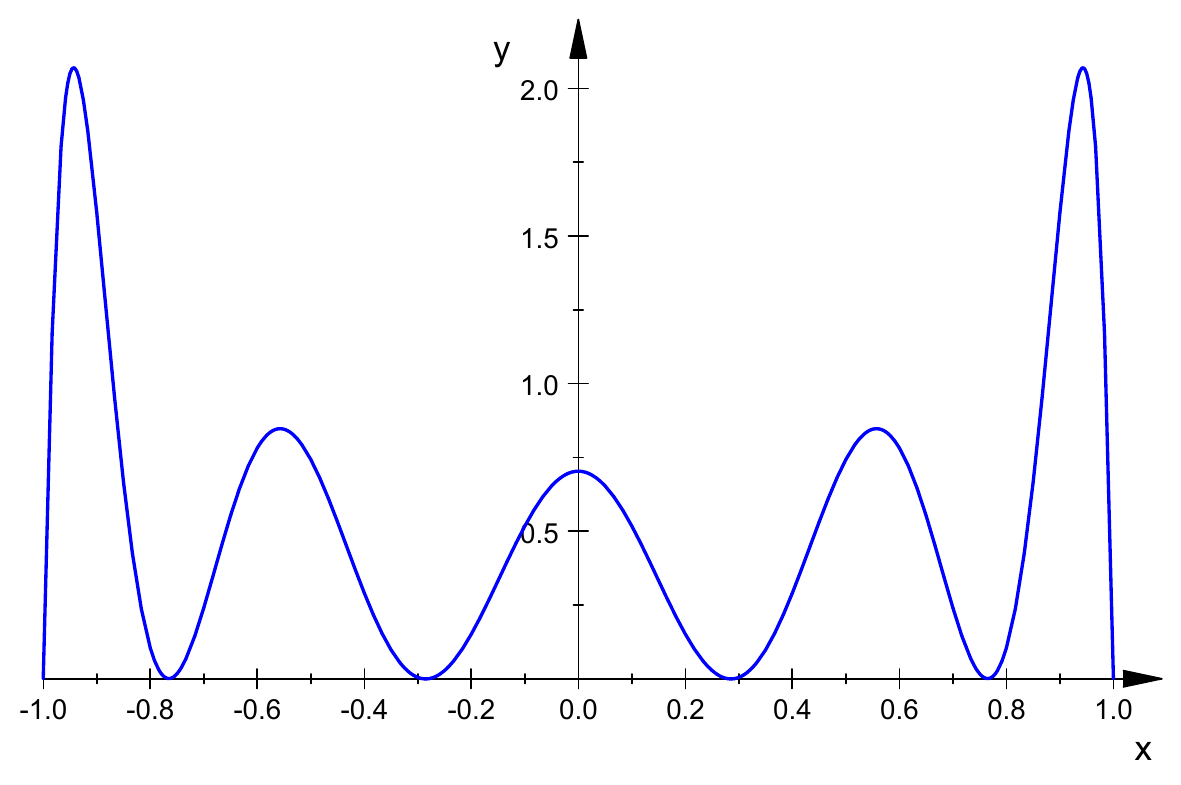}
\caption{Polynomial $p^\star$ for Example \ref{expl1}.}
\label{fig:Ch1}
\end{figure}

We observe that we get less points when using Problem \eqref{sdp-three} to recover the support for this example. This may occur due to numerical issues.

\subsection{Wynn's polygon}\label{expl2}
As a {first} two-dimensional example we take the polygon given by the vertices $(-1,-1),\ (-1,1),\ (1,-1)$ and $(2,2)$, scaled to fit the unit circle, {\it i.e.,} we consider the design space 
\[
\mathcal{X} = \{x\in\R^2 : x_1, x_2 \geqslant-\tfrac{1}{4}\sqrt{2},\ x_1\leq\tfrac{1}{3}(x_2+\sqrt{2}),\ x_2\leq\tfrac{1}{3}(x_1+\sqrt{2}),\ x_1^2+x_2^2\leq1\}.
\]
Note that we need the redundant constraint $x_1^2+x_2^2\leq1$ in order to have an algebraic certificate of compactness.

%\subsubsection{D-optimal design}
As before, in order to find the $D$-optimal measure for the regression, we solve Problems \eqref{sdp} and \eqref{sdp-second}. Let us start by analyzing the results for $d=1$ and $\delta=3$. Solving \eqref{sdp} we obtain $\y^\star\in\R^{45}$ which leads to 4 atoms when solving \eqref{sdp-second} with $r=3$. For the latter the moment matrices of order 2 and 3 both have rank 4, so Condition~\eqref{test} is fulfilled. As expected, the 4 atoms are exactly the vertices of the polygon.

Again, we could also solve Problem \eqref{sdp-four} instead of \eqref{sdp-second} to receive the same atoms. As in the univariate example we get less points when using Problem~\eqref{sdp-three}. To be precise, GloptiPoly is not able to extract any solutions for this example.

For increasing $d$, we get an optimal measure with a larger support. For $d=2$ we recover 7 points, and 13 for~$d=3$. See Figure \ref{fig:wynn} for the polygon, the supporting points of the optimal measure and the 
$\binom{2+d}{2}$-level set of the Christoffel polynomial $p^\star_d$ for different $d$. The latter demonstrates graphically that the set of zeros of $\binom{2+d}{d}-p^\star_d$ intersected with $\mathcal{X}$ are indeed the atoms of our representing measure. {In the picture the size of the support points is chosen with respect to their corresponding weights, i.e., the larger the point, the bigger the respective weight.

The numerical values of the support points and their weights computed in the above procedure (and displayed in Figure \ref{fig:wynn}) are listed in Appendix \ref{table}.

To get an idea of how the Christoffel polynomial looks like, we plot in Figure \ref{fig:wynnCh} the 3D-plot of the polynomial 
$-{p^\star}=p_d^\star-\binom{2+d}{2}$. This illustrates very clearly that the zeros of $p^\star$ on $\mathcal X$ are the support points of the optimal design.}

%In Figure \ref{fig:wynnweights} we visualized the weights corresponding to each point of the support for the different $d$.

\begin{figure}[ht]
\includegraphics[width=.32\textwidth]{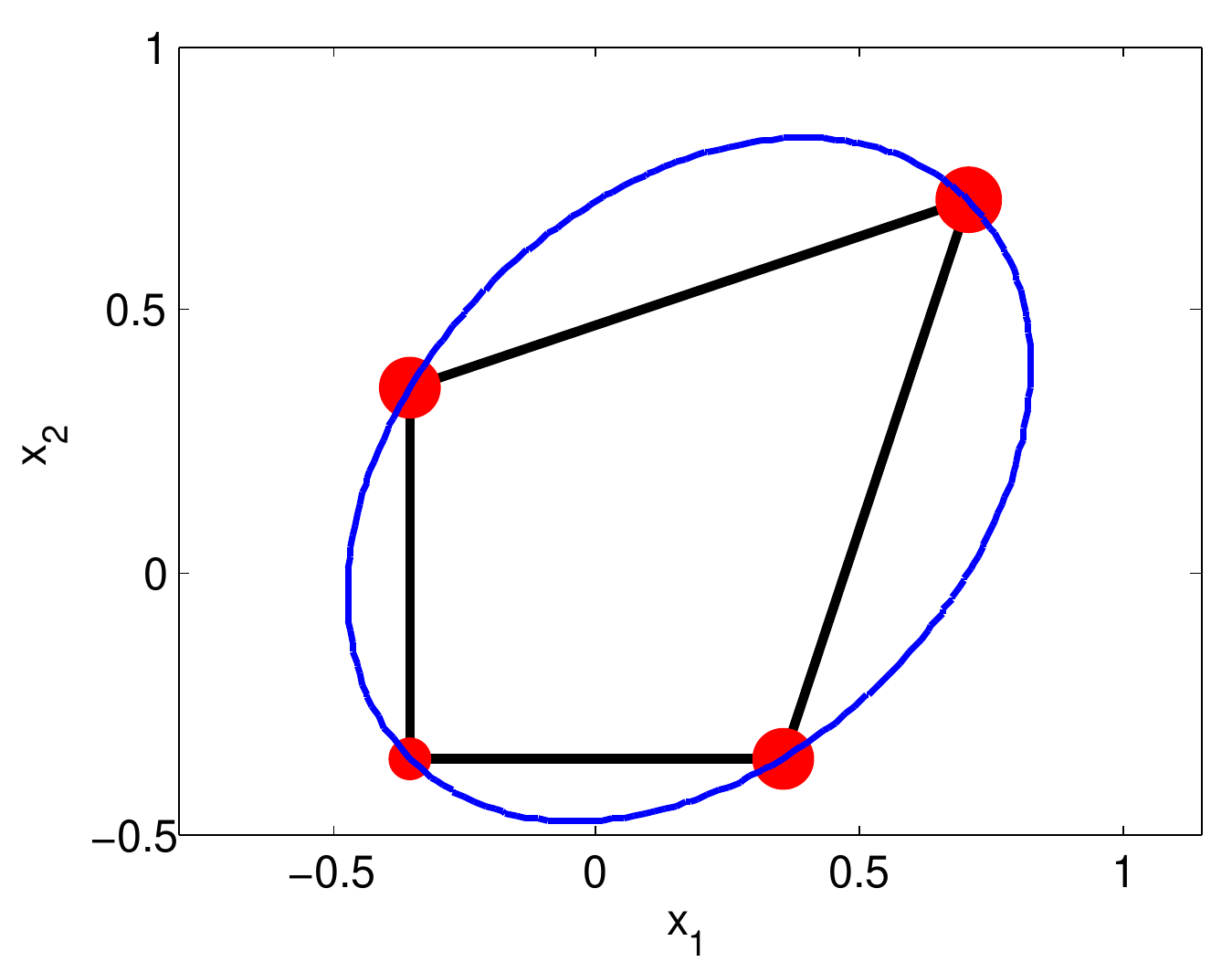}
\includegraphics[width=.32\textwidth]{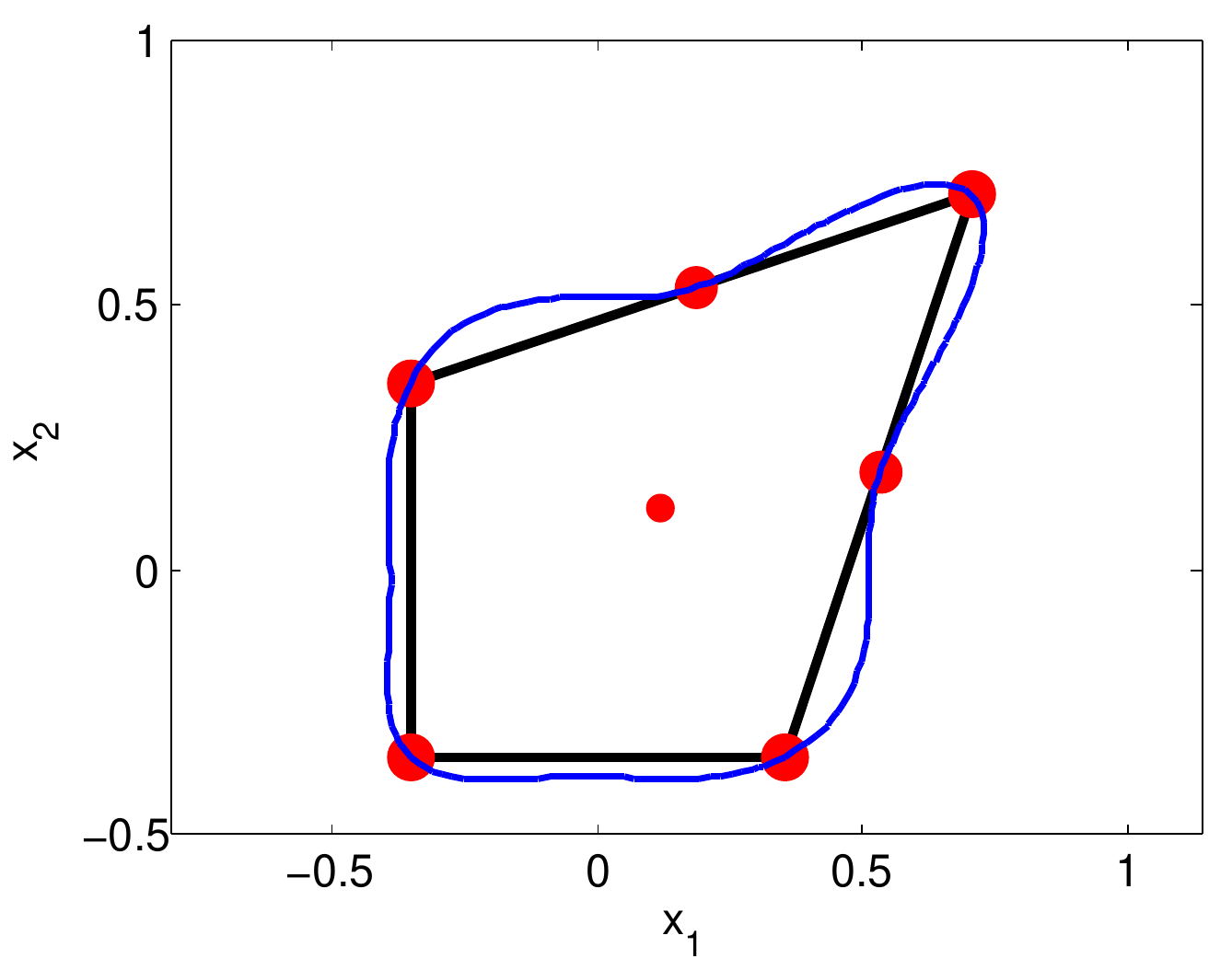}
\includegraphics[width=.32\textwidth]{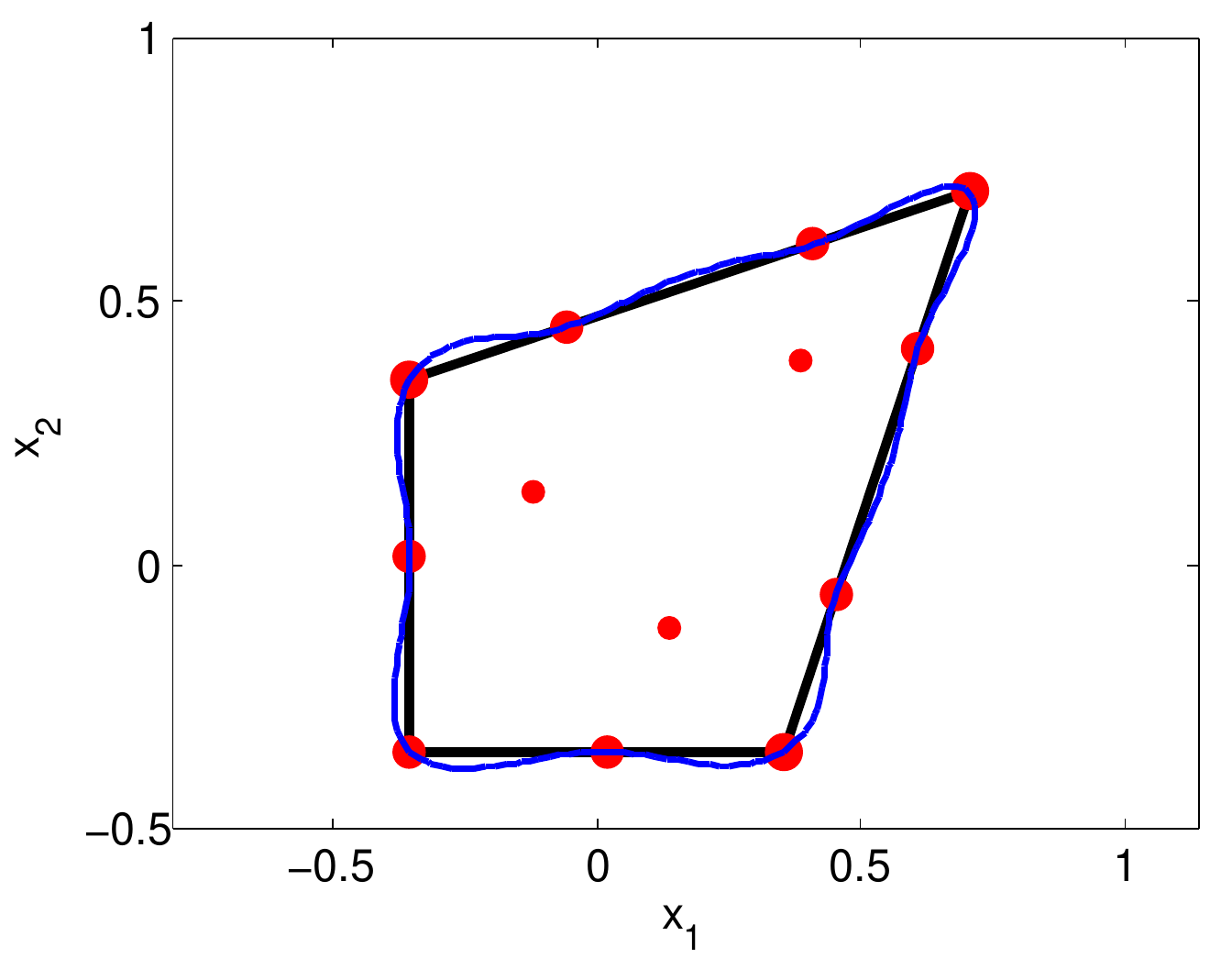}
\caption{The polygon (bold black) of Example \ref{expl2}, the support of the optimal design measure (red points) where the size of the points corresponds to the respective weights, and the $\binom{2+d}{2}$-level set of the Christoffel polynomial (thin blue) for $d=1$ (left), $d=2$ (middle), $d=3$ (right) and $\delta=3$.}
\label{fig:wynn}
\end{figure}

\begin{figure}[ht]
	\includegraphics[width=.33\textwidth]{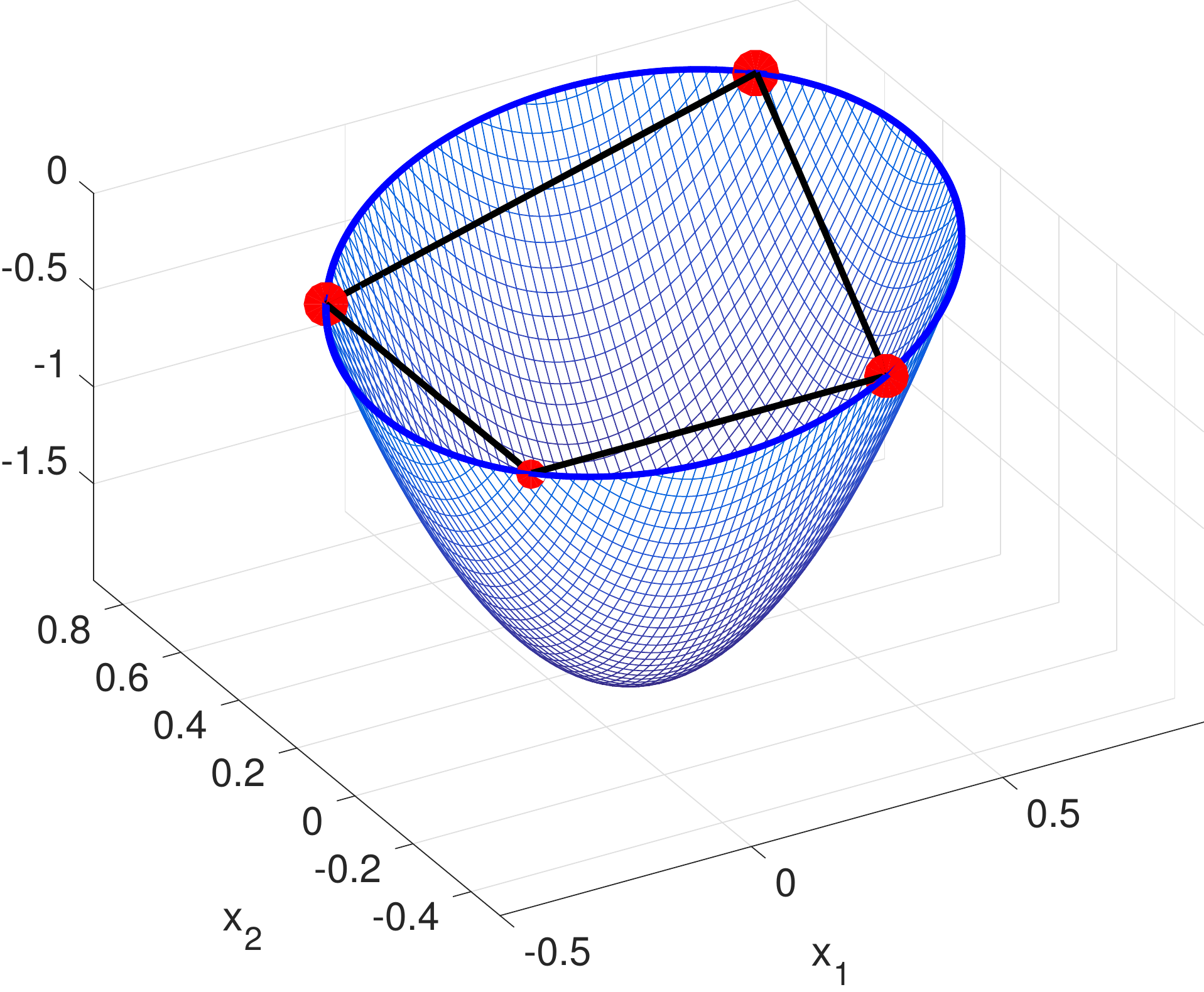}\hfill
	\includegraphics[width=.33\textwidth]{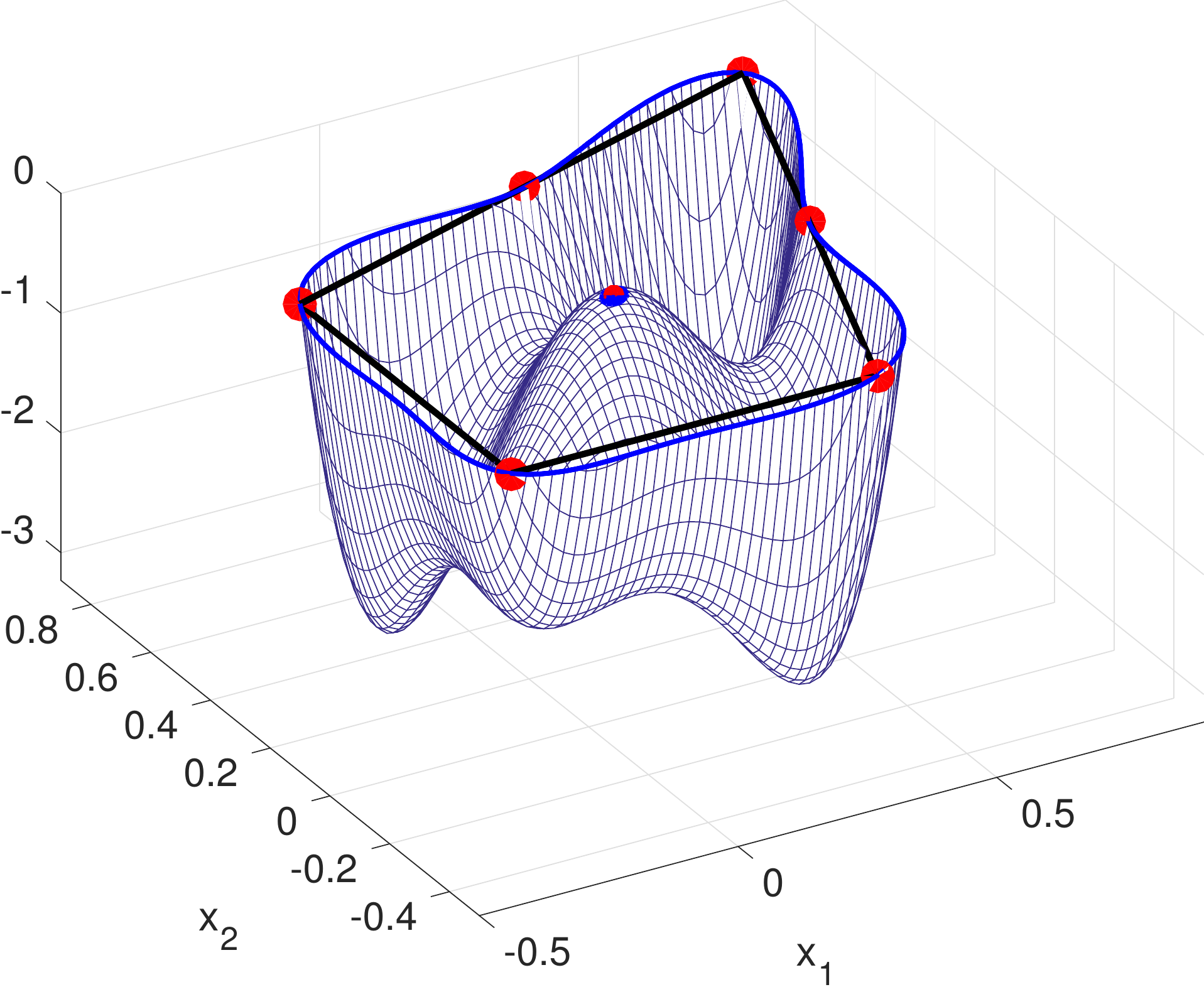}
	\includegraphics[width=.33\textwidth]{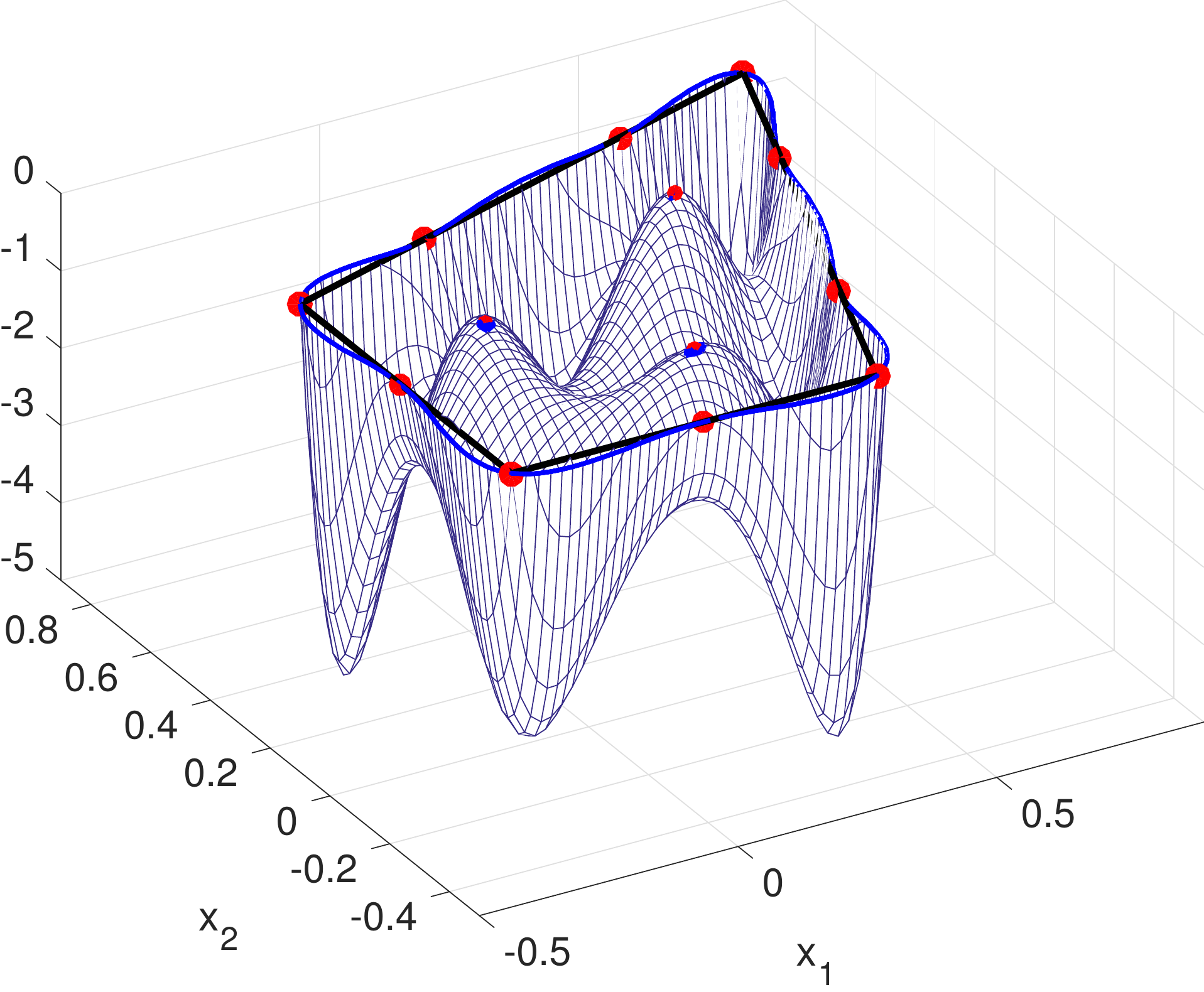}
	\caption{The polynomial $p_d^\star-\binom{2+d}{2}$  where $p_d^\star$ denotes the Christoffel polynomial of Example \ref{expl2} for $d=1$ (top left), $d=2$ (top right), $d=3$ (bottom middle). The red points correspond to the $\binom{2+d}{2}$-level set of the Christoffel polynomial.}
		\label{fig:wynnCh}
	\end{figure}

\subsection{Ring of ellipses}\label{expl2a}

As a second example in the plane we consider an ellipsoidal ring, i.e., an ellipse with a hole in the form of a smaller ellipse. More precisely,
\[
\mathcal{X} = \{x\in\R^2 : 9x_1^2+13x_2^2\leq 7.3,\ 5x_1^2+13x_2^2\geq 2\}.
\]
We follow the same procedure as described in the former example. See Figure~\ref{fig:ellipses} for the results. The values are again listed in Appendix \ref{table}.

\begin{figure}[h]
\includegraphics[width=.32\textwidth]{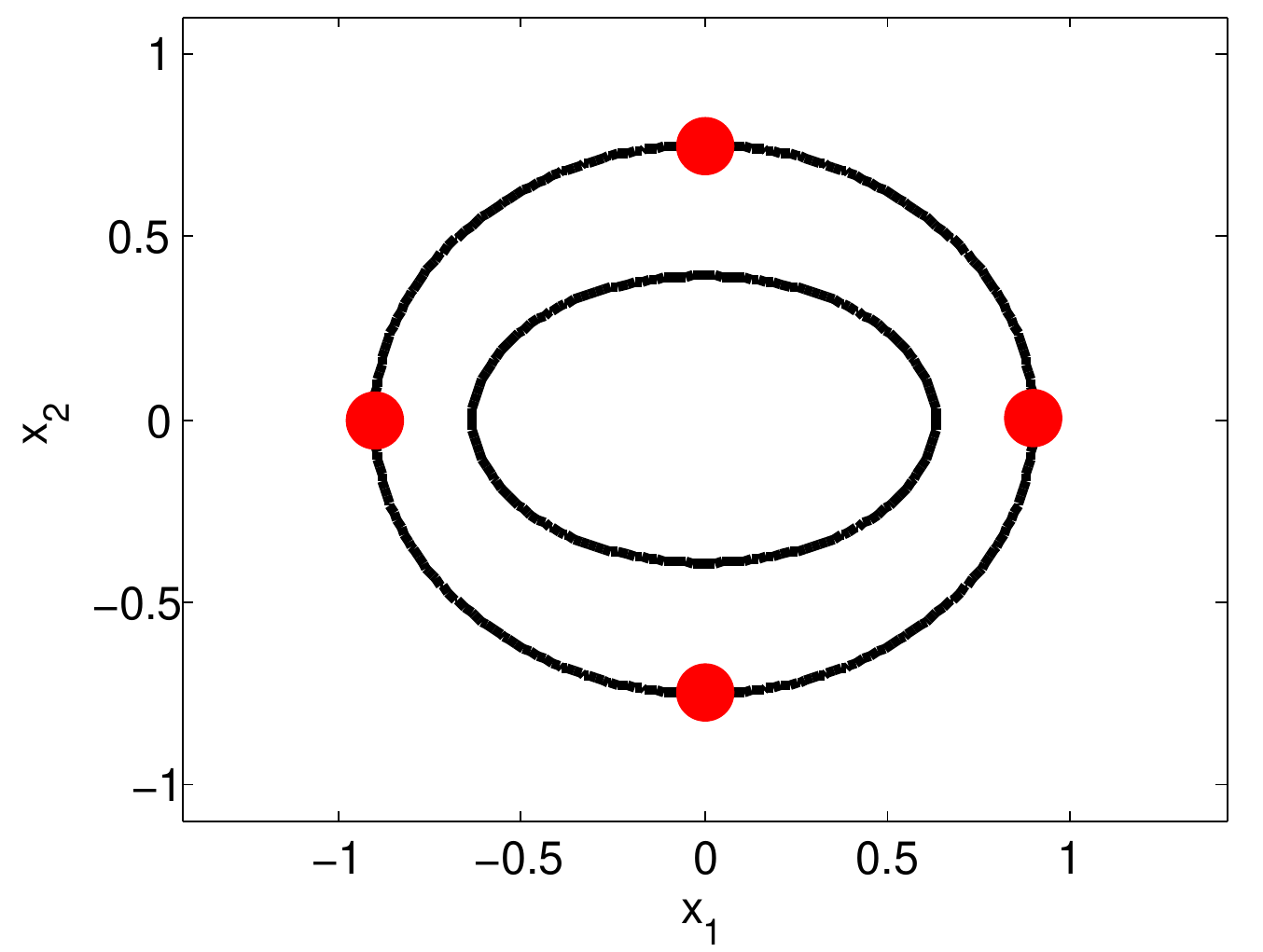}
\includegraphics[width=.32\textwidth]{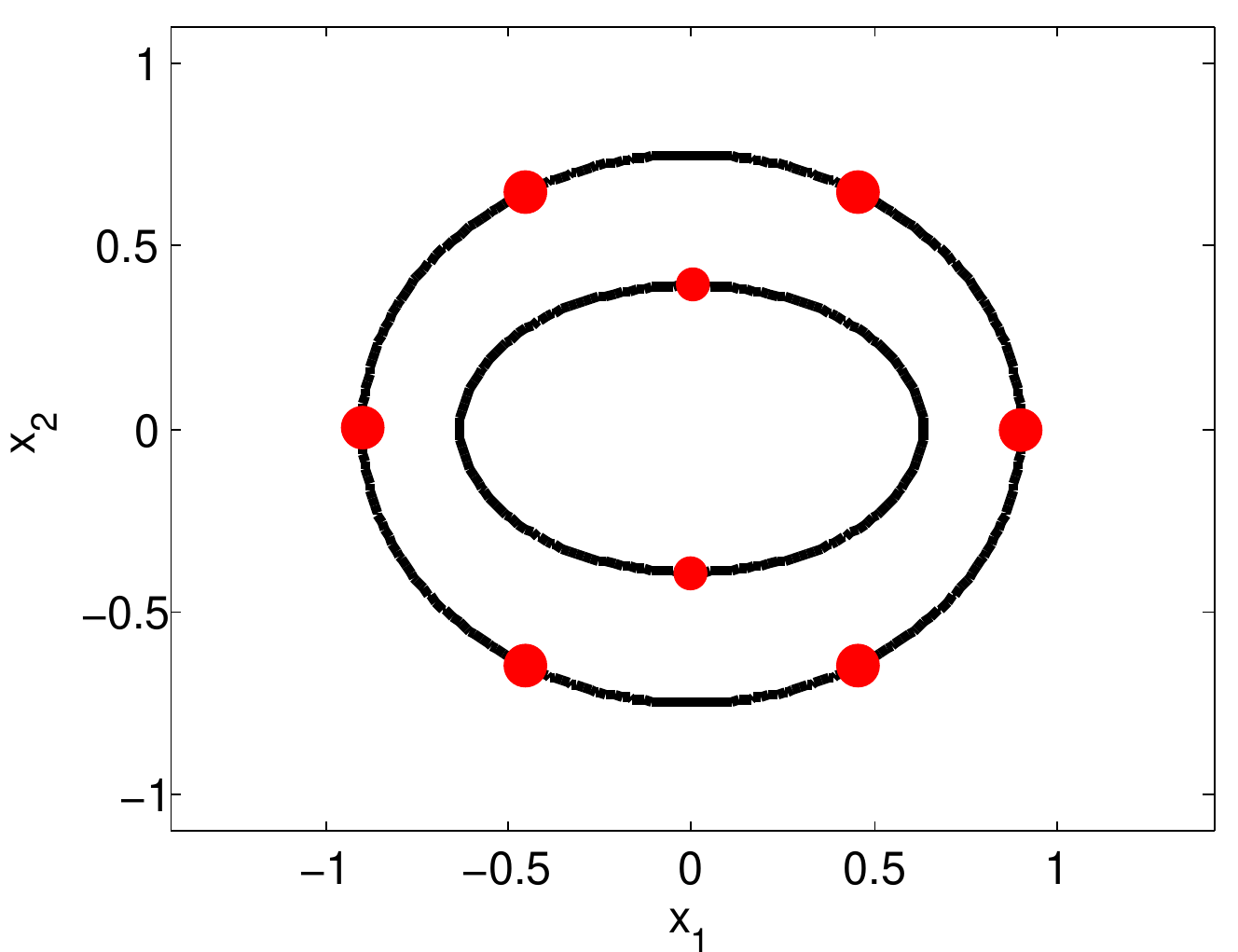}
\includegraphics[width=.32\textwidth]{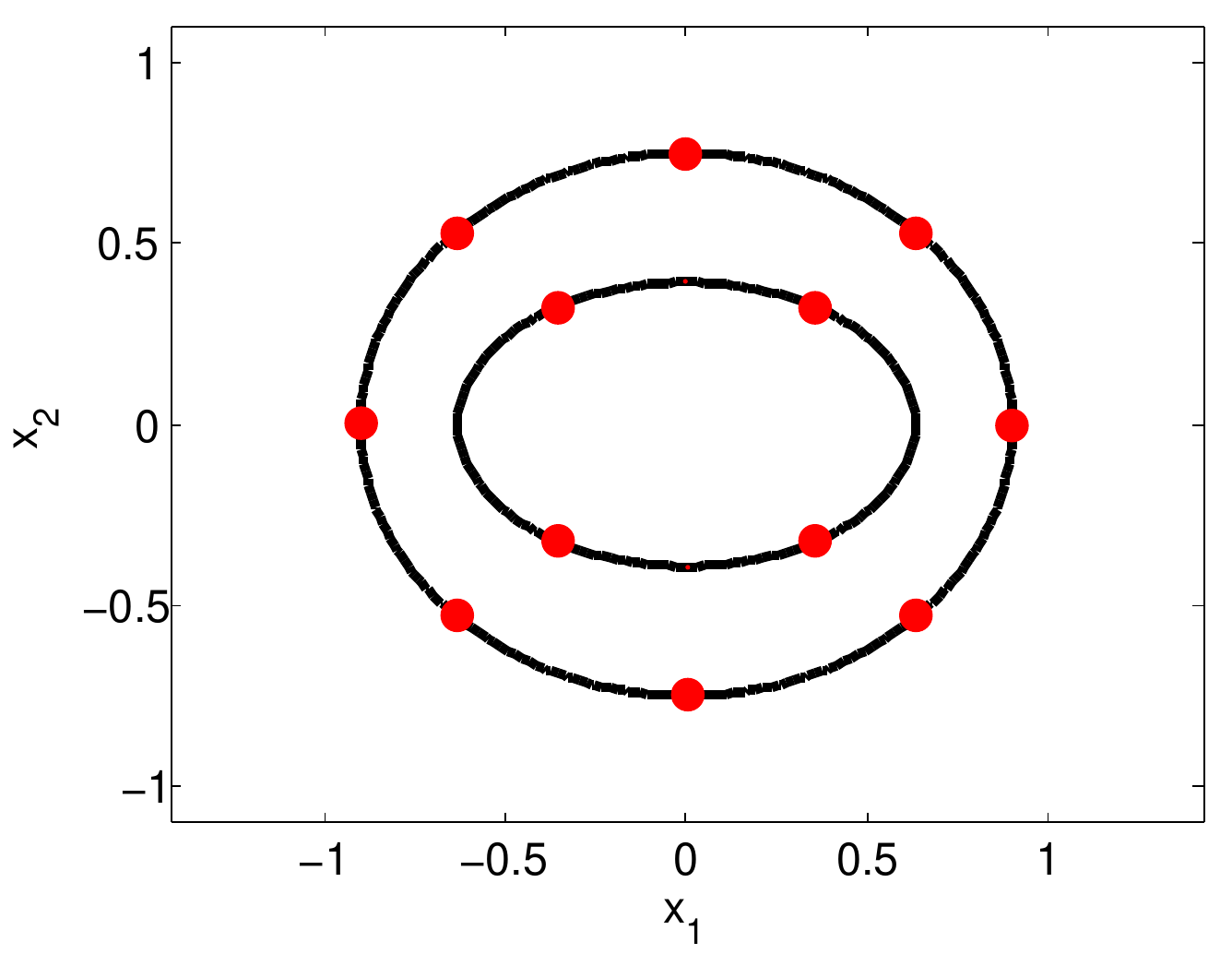}
\caption{The polygon (bold black) of Example \ref{expl2a} and the support of the optimal design measure (red points) where the size of the points corresponds to the respective weights for $d=1$ (left), $d=2$ (middle), $d=3$ (right) and $\delta=3$.}
\label{fig:ellipses}
\end{figure}

\subsection{Moon}\label{expl2b}

To investigate another non-convex example, we apply our method to the moon-shaped semi-algebraic set
\[
\mathcal{X} = \{x\in\R^2 : (x_1+0.2)^2+x_2^2\leq 0.36,\ (x_1-0.6)^2+x_2^2\geq 0.16\}.
\]
The results are represented in Figure \ref{fig:moon} and for the numerical values the interested reader is referred to Appendix \ref{table}.

\begin{figure}[ht]
\includegraphics[width=.32\textwidth]{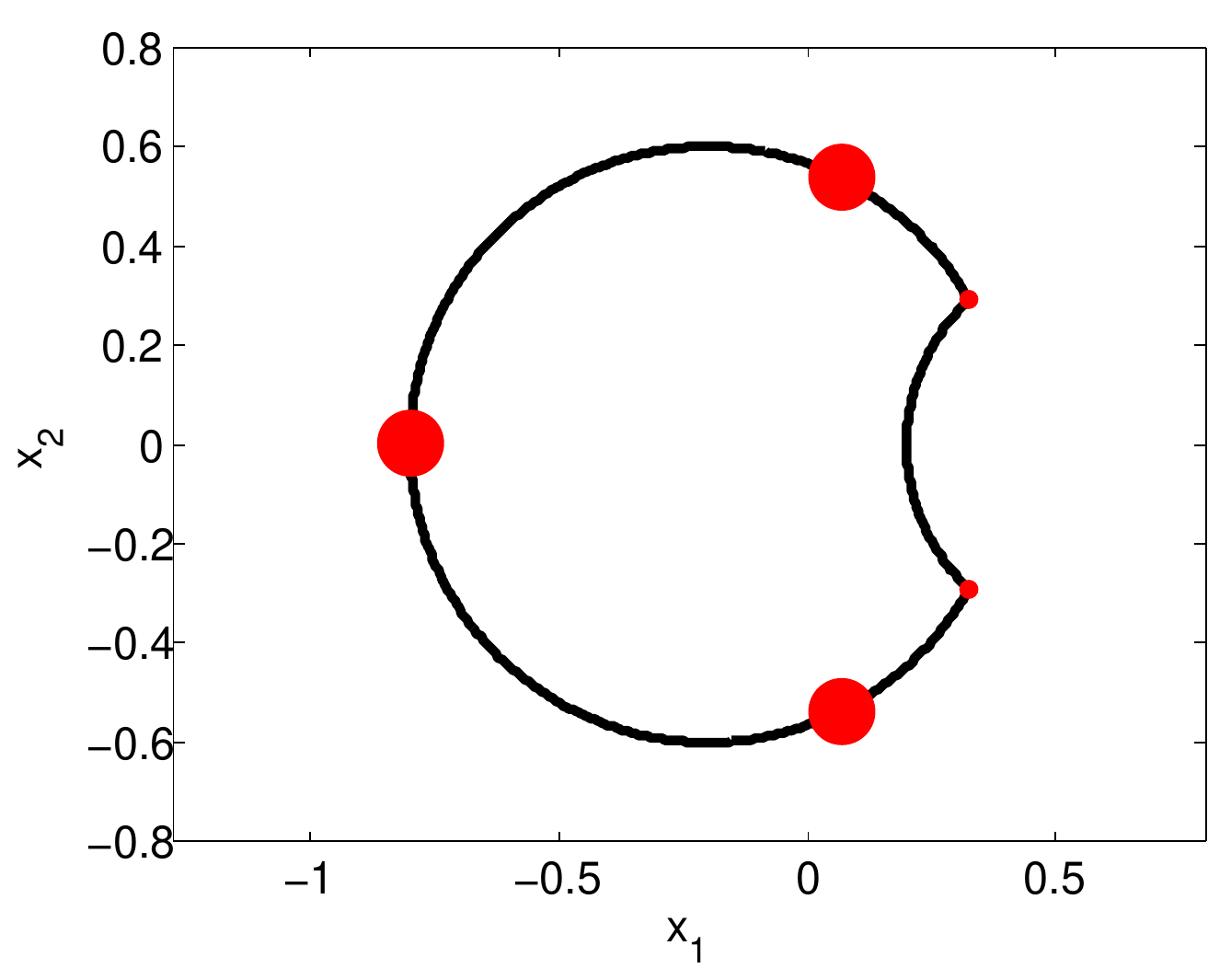}
\includegraphics[width=.32\textwidth]{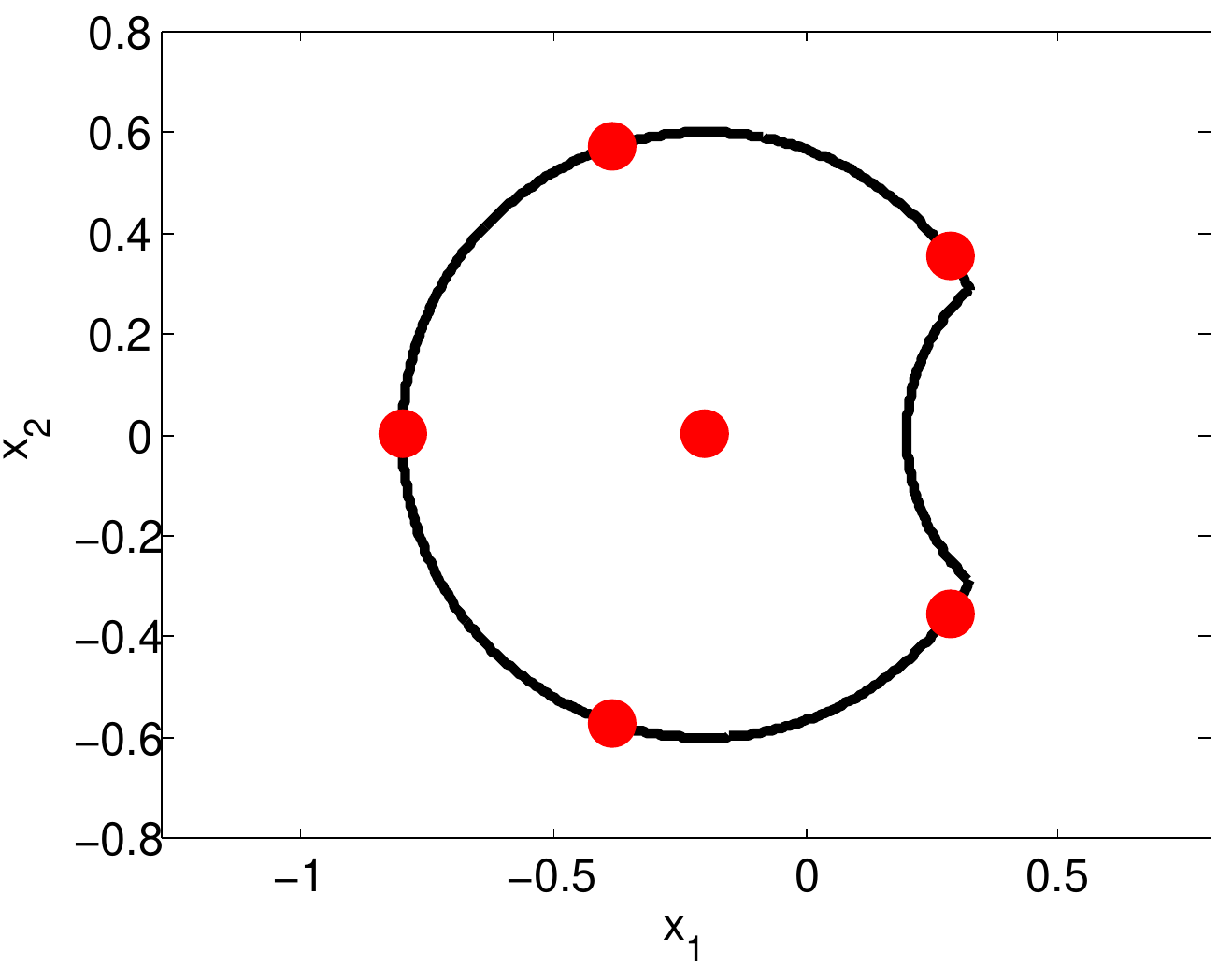}
\includegraphics[width=.32\textwidth]{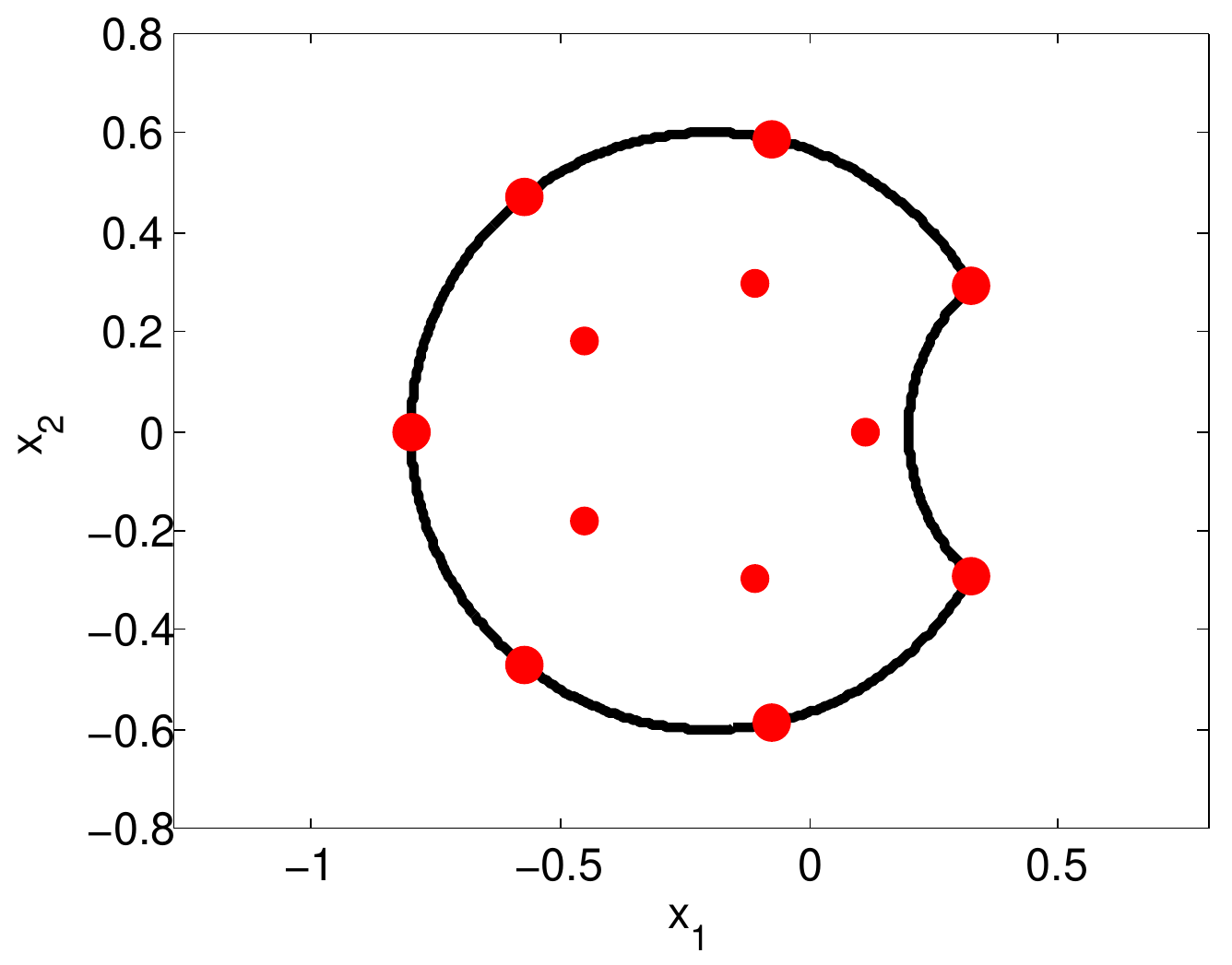}
\caption{The polygon (bold black) of Example \ref{expl2b} and the support of the optimal design measure (red points) where the size of the points corresponds to the respective weights for $d=1$ (left), $d=2$ (middle), $d=3$ (right) and $\delta=3$.}
\label{fig:moon}
\end{figure}

\subsection{Folium}\label{expl2c}

The zero level set of the polynomial $f(x) = -x_1(x_1^2-2x_2^2)\-(x_1^2+x_2^2)^2$ is a curve of genus zero with a triple singular point at the origin. It is called a folium. As a last two-dimensional example we consider the semi-algebraic set defined by $f$, i.e.,
\[
\mathcal{X} = \{x\in\R^2 : f(x)\geq0,\ x_1^2+x_2^2\leq1\}.
\]
Figure \ref{fig:folium} illustrates the results and the values are listed in Appendix \ref{table}.

\begin{figure}[ht]
\includegraphics[width=.32\textwidth]{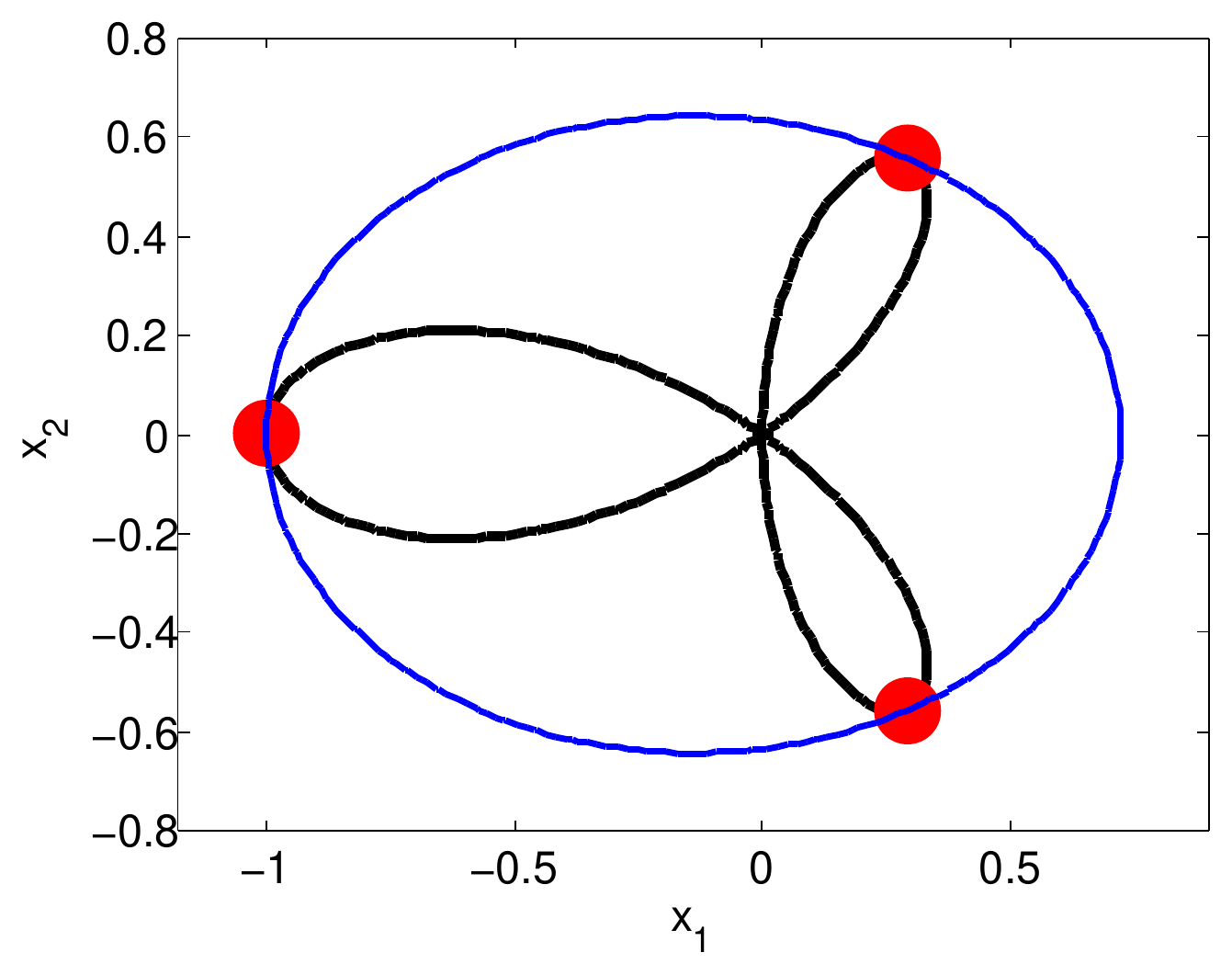}
\includegraphics[width=.32\textwidth]{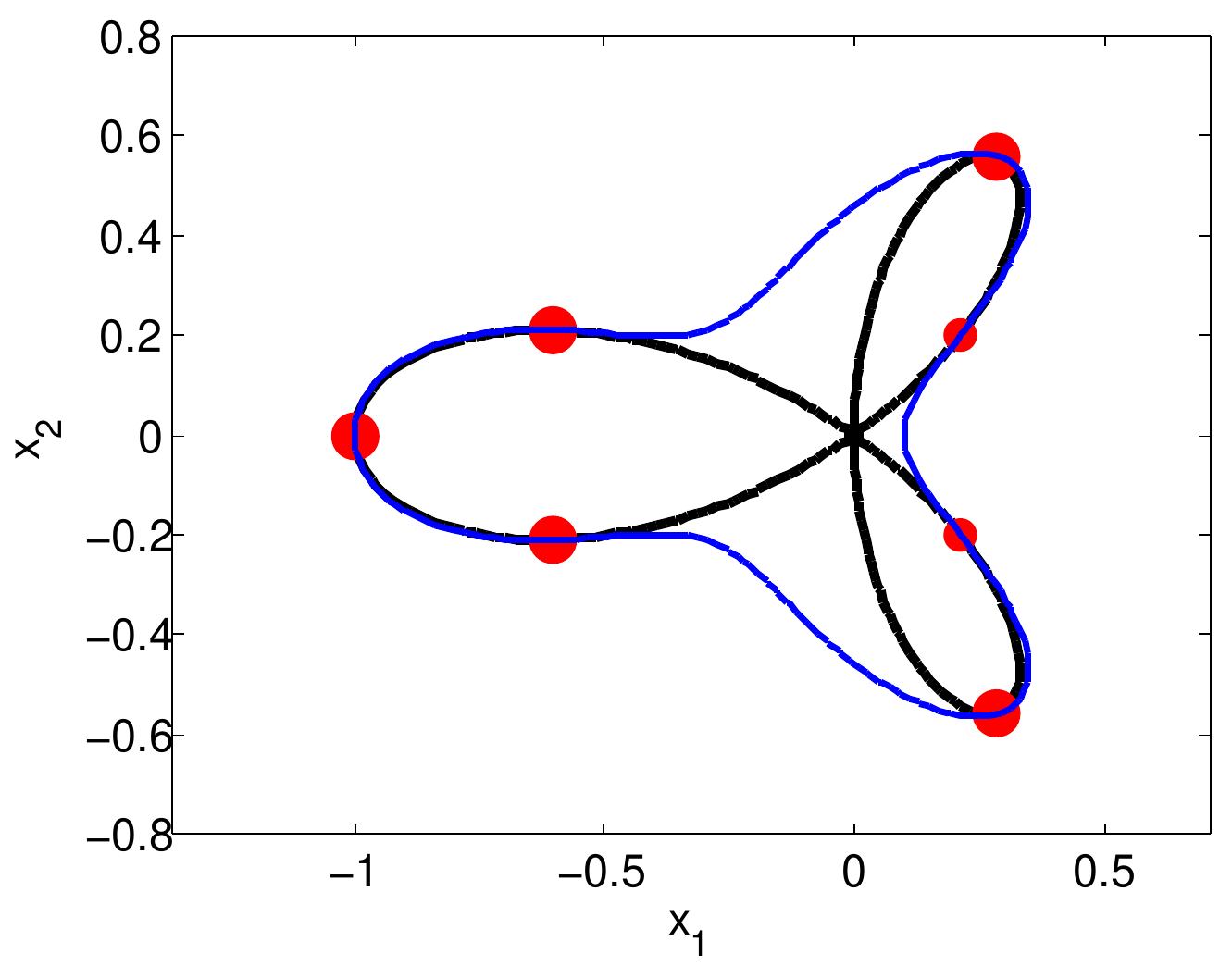}
\includegraphics[width=.32\textwidth]{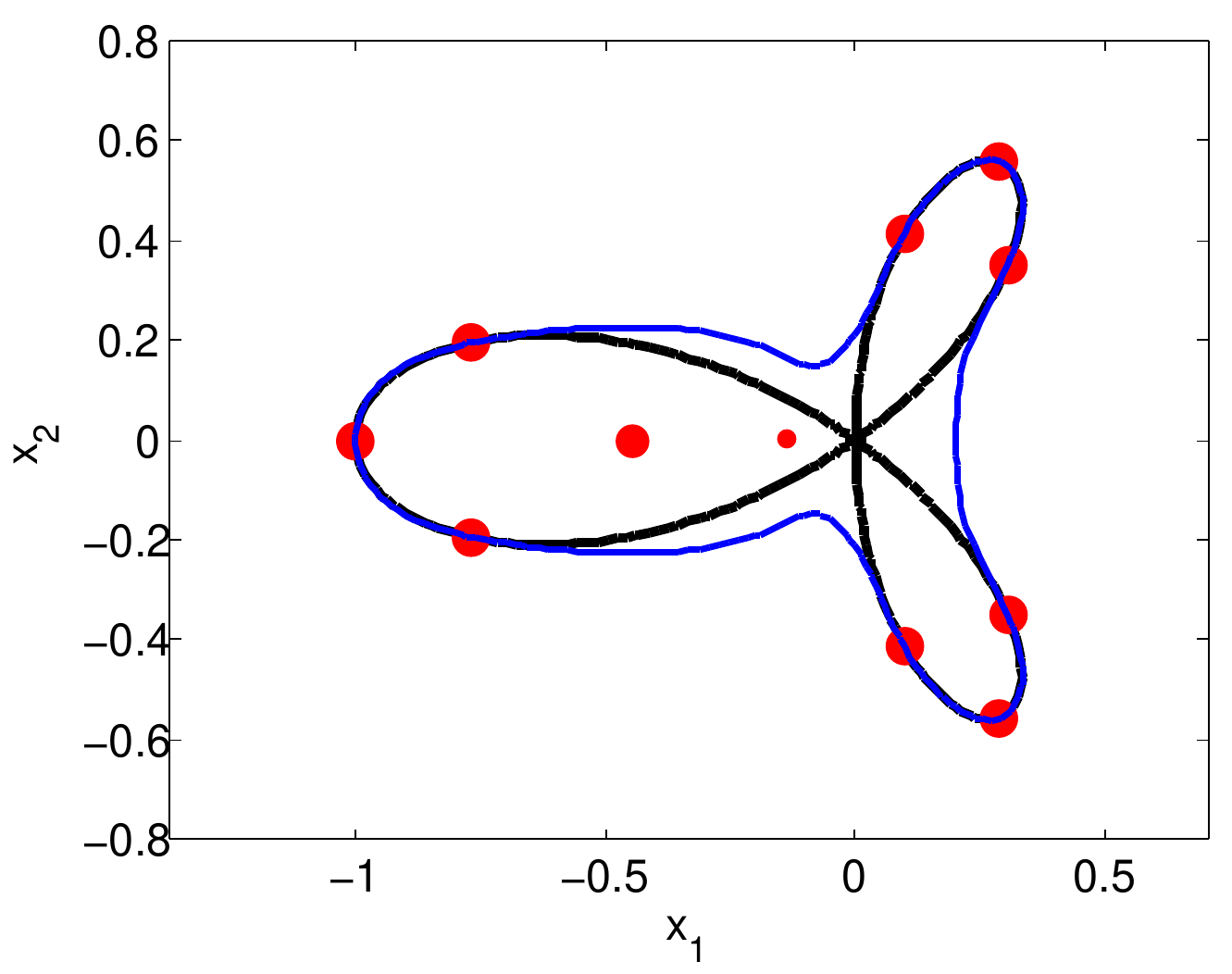}
\caption{The polygon (bold black) of Example \ref{expl2c}, the support of the optimal design measure (red points) where the size of the points corresponds to the respective weights, and the $\binom{2+d}{2}$-level set of the Christoffel polynomial (thin blue) for $d=1$ (left), $d=2$ (middle), $d=3$ (right) and $\delta=3$.}
\label{fig:folium}
\end{figure}

\subsection{The 3-dimensional unit sphere}\label{expl3}
Last, let us consider the regression for the degree $d$ polynomial measurements $\sum_{|\alpha|\leq d} \theta_\alpha x^\alpha$ on the unit sphere $\mathcal{X}=\{x\in\R^3: x_1^2+x_2^2+x_3^2=1\}$. Again, we first solve Problem \eqref{sdp}. For $d=1$ and $\delta\geq0$ we obtain the sequence $\y^\star\in\R^{10}$ with $y_{000}^\star=1,\ y_{200}^\star=y_{020}^\star=y_{002}^\star=0.333$ and all other entries zero.

In the second step we solve Problem \eqref{sdp-second} to recover the measure. For $r=2$ the moment matrices of order 2 and 3 both have rank 6, meaning the rank condition \eqref{test} is fulfilled, and we obtain the six atoms $\{(\pm1,0,0),(0,\pm1,0),(0,0,\pm1)\}\subseteq\mathcal{X}$ on which the optimal measure~$\mu\in\Ms_+(\mathcal{X})$ is uniformly supported.

\pagebreak[3]

For quadratic regressions, {\it i.e.,} $d=2$, we obtain an optimal measure supported on 14 atoms evenly distributed on the sphere. Choosing $d=3$, meaning cubic regressions, we find a Dirac measure supported on 26 points which again are evenly distributed on the sphere. See Figure \ref{fig:sphere} for an illustration of the supporting points of the optimal measures for $d=1$, $d=2$, $d=3$ and $\delta=0$.

\begin{figure}[ht]
\includegraphics[width=.32\textwidth]{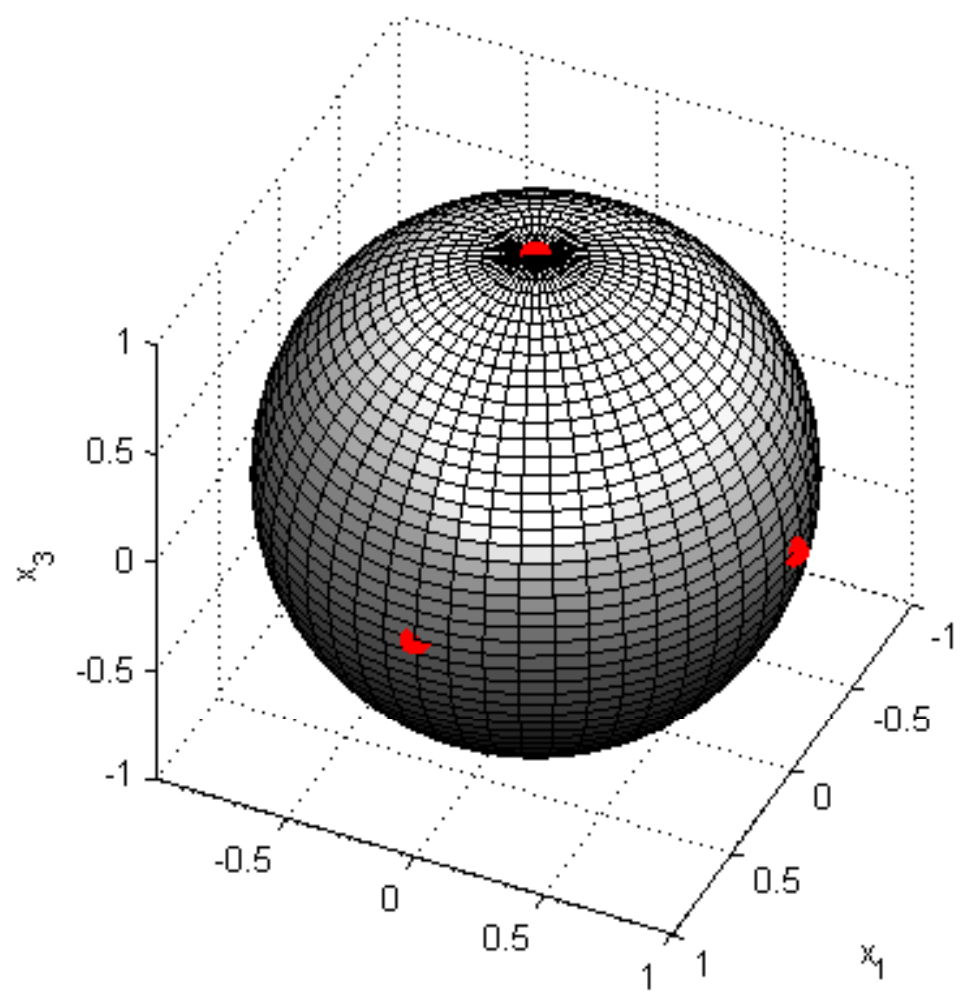}
\includegraphics[width=.32\textwidth]{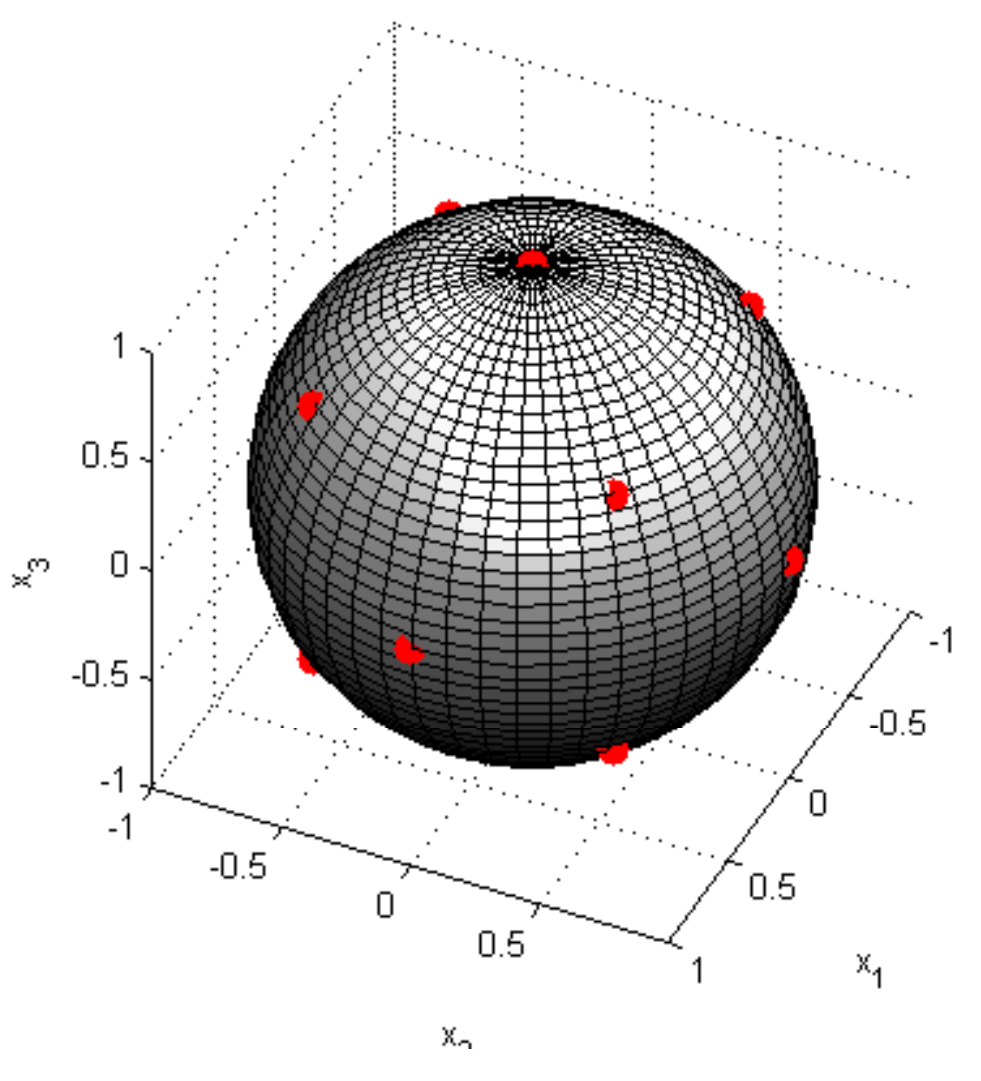}
\includegraphics[width=.32\textwidth]{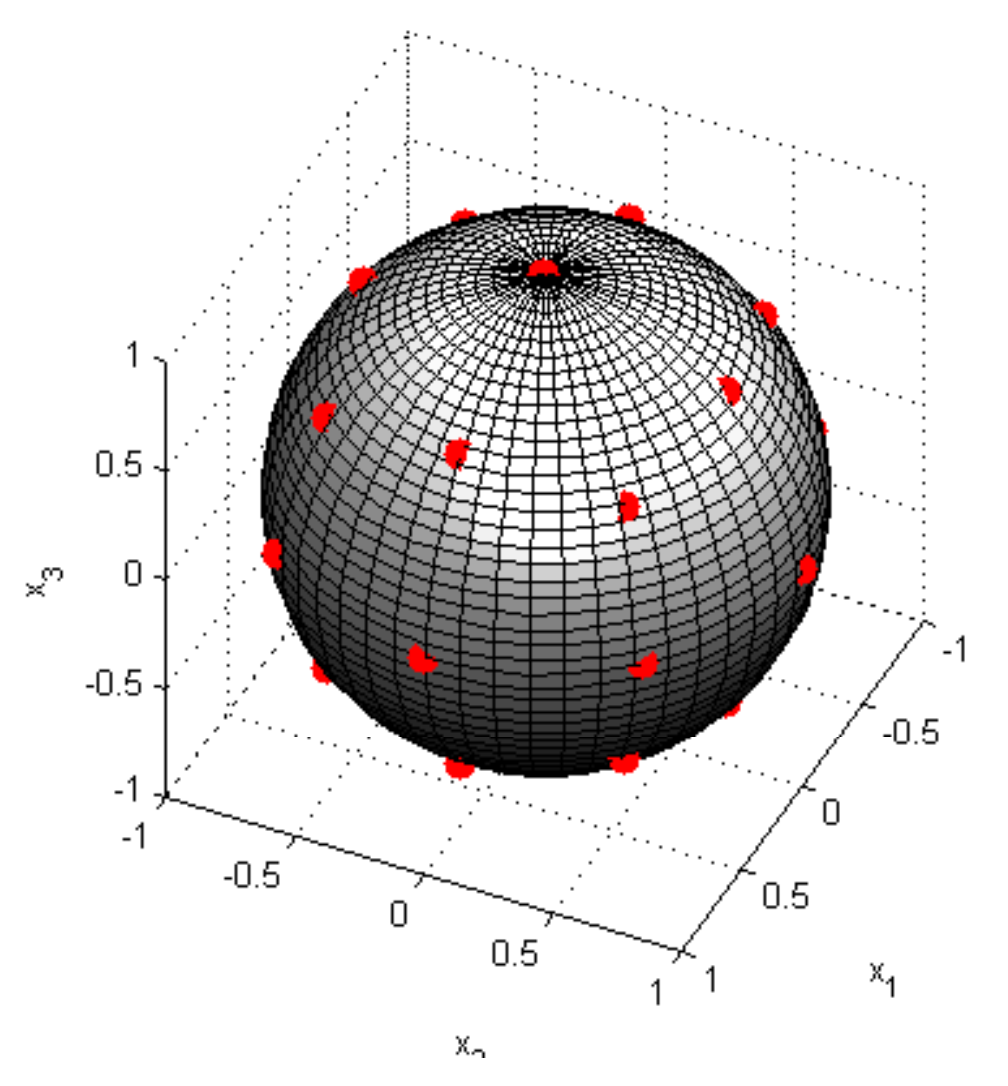}
\caption{The red points illustrate the support of the optimal design measure for $d=1$ (left), $d=2$ (middle), $d=3$ (right) and $\delta=0$ for Example \ref{expl3}.}
\label{fig:sphere}
\end{figure}

Using the method via Christoffel polynomials gives again less points. No solution is extracted when solving Problem \eqref{sdp-four} and we find only two supporting points for Problem \eqref{sdp-three}.

\begin{figure}[ht]
\includegraphics[width=0.4\textwidth]{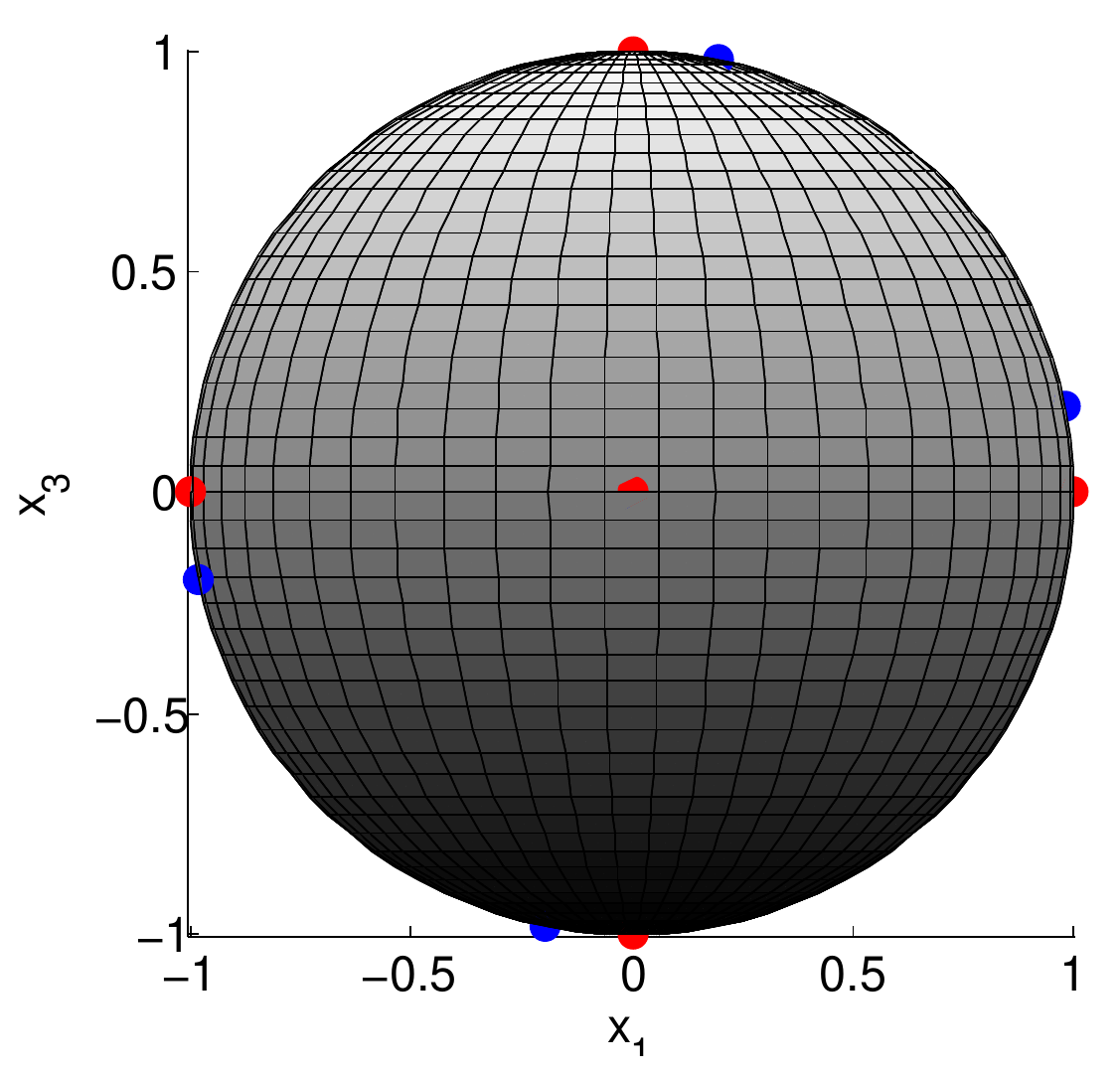}
\caption{Support points recovered in Example \ref{expl3} for the $D$-optimal design and $d=1$ (red) and the points which are recovered when additionally fixing some moments as described in Subsection \ref{fixed} (blue).}
\label{fig:FixMom}
\end{figure}

\subsection{Fixing some moments}\label{fixed}
Our method has an additional nice feature. Indeed in Problem \eqref{sdp} one may easily include the additional constraint that some moments $(y_\alpha)$, $\alpha\in\Gamma\subset\N^n_{2d}$ are fixed to some prescribed value. We illustrate this potential on one example.
For instance, with $\Gamma=\{(020), (002), (110), (101)\}$, let  $y_{020}:=2$, $y_{002}:=1$, $y_{110}:=0.01$ and $y_{101}:=0.95$. In order to obtain a feasible problem, we scale them with respect to the Gauss distribution.

For the $D$-optimal design case with $d=1$ and $\delta=0$ and after computing the support of the corresponding measure using the Nie method, we get 6 points as we obtain without fixing the moments. However, now~four of the six points are shifted and the measure is no longer uniformly supported on these points, but each two opposite points have the same weight. See Figure~\ref{fig:FixMom} for an illustration of the position of the points with fixed moments (blue) with respect to the position of the support points without fixing the points (red).

\section{Conclusion}
In this paper, we give a general method to build optimal designs for multidimensional polynomial regression on an algebraic manifold. The method is highly versatile as it can be used for all classical functionals of the information matrix. Furthermore, it can easily be tailored to incorporate prior knowledge on some multidimensional moments of the targeted optimal measure (as proposed in \cite{molchal04}). In future works, we will extend the method to multi-response polynomial regression problems and to general smooth parametric regression models by linearization.

\subsection*{Acknowledgments}
{We warmly thank three anonymous referees for their valuable comments on early version of this paper.}
We thank Henry Wynn for communicating the polygon of Example \ref{expl2} to us.
Feedback from Pierre Mar\'echal, Luc Pronzato, Lieven Vandenberghe and Weng Kee Wong was also appreciated.

The research of the last three authors is funded by the European Research Council (ERC) under the European Union’s Horizon 2020 research and innovation program (grant agreement ERC-ADG 666981 TAMING).

\appendix
\section{Proof of Theorem~1}
\label{proof:thm-ideal}

First, let us prove that Problem \eqref{sdp-ideal} has an optimal solution. The feasible set is nonempty with finite associated objective value\textemdash take as feasible point the vector $\y\in \mathcal{M}_{2d}(\K)$ associated with the Lebesgue  measure on the compact set $\K$, scaled to be a probability measure. Moreover, as~$\mathcal{X}$ is compact with nonempty interior, it follows that $\mathcal{M}_{2d}(\mathcal{X})$ is closed (as the dual of~$\mathcal{P}_{2d}(\mathcal{X})$).

In addition, the feasible set $\{\y\in\mathcal{M}_{2d}(\mathcal{X}): y_0=1\}$ of Problem~\eqref{sdp-ideal} is compact. Indeed there exists $M>1$ such that it holds $\int_\K x_i^{2d}\,d\mu< M$ for every probability measure~$\mu$ on~$\K$ and every $i=1,\ldots,n$. Hence, $\max\{y_0,\ \max_i\{L_y(x_i^{2d})\}\}<M$ which by \cite{lass-netzer} implies that $\vert y_\alpha\vert\leq M$ for every $|\alpha|\leq 2d$, which in turn implies that the feasible set of \eqref{sdp-ideal} is compact.

Next, as the function $\phi_q$ is upper semi-continuous, the supremum in \eqref{sdp-ideal} is attained at some optimal solution $\y^\star\in\mathcal{M}_{2d}(\mathcal{X})$. Moreover, as the feasible set is convex %, when 
{and} $\phi_q$ is strictly concave {(see, e.g., \cite[Chapter~6.13]{pukelsheim2006optimal})} then $\y^*$ is the unique optimal solution.

{Now, we examine the properties of the polynomial $p^\star$ and show the equivalence statement. For this we notice that there exists a strictly feasible solution because the cone ${\rm int}(\mathcal{M}_{2d}(\mathcal{X}))$ is nonempty by Lemma 2.6 in~\cite{malin15}. Hence, Slater's condition\footnote{For the optimization problem $\max\,\{f(\x): A\x=b;\ \x\in C\}$, where $A\in\R^{m\times n}$ and $C\subseteq\R^n$ is a nonempty closed convex cone, Slater's condition holds, if there exists a feasible solution $\x$ in the interior of $C$.}
holds for \eqref{sdp-ideal}. Further, by a an argument in \cite[Chapter 7.13]{pukelsheim2006optimal}, the matrix $M_d(\y^\star)$ is non-singular. Therefore, $\phi_q$ is differentiable at~$\y^\star$. Since additionally Slater's condition is fulfilled and $\phi_q$ is concave, this implies that %$\y^\star$ satisfies the necessary 
the {\it Karush-Kuhn-Tucker} (KKT) optimality conditions\footnote{For the optimization problem $\max\,\{f(\x): A\x=b;\ \x\in C\}$, where $f$ is differentiable, $A\in\R^{m\times n}$ and $C\subseteq\R^n$ is a nonempty closed convex cone, the KKT-optimality conditions at a feasible point $\x$ state that there exist $\lambda^\star\in \R^m$ and $u^\star\in C^\star$ such that $A^{\top}\lambda^\star-\nabla f(\x)=u^\star$ and $\langle \x,u^\star\rangle=0$.}
%\footnote{For the optimization problem $\min\,\{f(x): Ax=b;x\in C\}$ where $f$ is differentiable, $A\in\R^{m\times n}$ and $C\subset\R^n$ is a nonempty closed convex cone, the KKT-optimality conditions at a feasible point $x$ state that there exists $\lambda^\star\in \R^m$ and $u\in C^\star$ such that $\nabla f(x)-A^{\top}\lambda^\star=u^\star$ and $\langle x,u^\star\rangle=0$. Slater's condition holds if there exists a feasible solution $x\in{\rm int}(C)$, in which case the KKT-optimality conditions are necessary and sufficient if $f$ is convex.}.
at $\y^\star$ are necessary (and sufficient) for $\y^\star$ to be an optimal solution.

The KKT-optimality conditions read
\[
%\lambda^\star\,e_0-\nabla\phi_q(\M_d(\y^\star))\,=\,p^\star\in\mathcal{M}_{2d}(\mathcal{X})^\star\:(=\mathcal{P}_{2d}(\mathcal{X})),
\lambda^\star\,e_0-\nabla\phi_q(\M_d(\y^\star))\,=\,\hat{\mathbf{p}}^\star \quad \text{with } \hat{p}^\star=\langle\hat{\mathbf{p}},\v_{2d}(x)\rangle\in\mathcal{M}_{2d}(\mathcal{X})^\star\:(=\mathcal{P}_{2d}(\mathcal{X})),
\]
(where $\hat{\mathbf{p}}\in\R^{\binom{n+2d}{n}}$, $e_0=(1,0,\ldots,0)$, and $\lambda^\star$ is the dual variable associated with the constraint $y^\star_0=1$). The complementarity condition is $\langle \y^{\star},p^\star \rangle=0$.}

Writing $\B_\alpha$, $\alpha\in\N^n_{2d}$, for the real symmetric matrices satisfying 
$$\forall x\in\K,\quad\sum_{|\alpha|\leq 2d}\B_\alpha x^\alpha = \v_d(\x)\v_d(\x)^\top\!,$$ 
and $\langle \mathbf{A},\B\rangle= {\rm trace} (\mathbf{A}\B)$ for two real symmetric matrices $\mathbf{A}$ and $\B$, this can be expressed as
{\begin{equation}
\label{a1-ideal}
\Big(1_{\alpha=0}\,\lambda^\star-\langle\nabla\phi_q(\M_d(\y^\star)),\B_\alpha\rangle\Big)_{|\alpha|\leq 2d}\,=\,\hat{\mathbf{p}},\quad \hat{p}^\star\,\in\,\mathcal{P}_{2d}(\K).%;\quad \langle \y^\star,p^\star\rangle =0.
\end{equation}}
Multiplying \eqref{a1-ideal} term-wise by $y^\star_\alpha$, summing up and invoking the complementarity condition, yields
\begin{align}
\label{lambdastar}
\lambda^\star\,=\,\lambda^\star\,y^\star_0\,
\overset{\eqref{a1-ideal}}{=}\,&\Big\langle\nabla\phi_q(\M_d(\y^\star)),\sum_{|\alpha|\leq 2d}y^\star_\alpha\B_\alpha\Big\rangle\,
\\=\,&\Big\langle\nabla\phi_q(\M_d(\y^\star)),\M_d(\y^\star)\Big\rangle\,
=\phi_q(\M_d(\y^\star))
\,,\notag
\end{align}
where the last equality holds by Euler formula for the positively homogeneous function $\phi_q$.

Similarly, multiplying Equation~\eqref{a1-ideal} term-wise by $\x^\alpha$ and summing up yields for all $\x\in\K$
\begin{align}
\label{pstar}
\x\mapsto {\hat{p}^\star(\x)\overset{\eqref{a1-ideal}}{=}}&\lambda^\star-\Big\langle\nabla\phi_q(\M_d(\y^\star)),\sum_{|\alpha|\leq 2d}\B_\alpha\x^\alpha\Big\rangle\\
=&\lambda^\star-\Big\langle\nabla\phi_q(\M_d(\y^\star)),\v_d(\x)\v_d(\x)^\top\Big\rangle
%=\lambda^\star-p_d^\star(\x)
\geq0.\notag
\end{align}

{For $q\neq 0$ let $c^\star:=\binom{n+d}{n}\Big[{\binom{n+d}{n}}^{-1}{\mathrm{trace}(\M_d(\y^\star)^{q})}\Big]^{1-\frac1q}$. As $\M_d(\y^\star)$ is positive semidefinite and non-singular, we have $c^\star>0$. 
If $q=0$, let $c^\star:=1$ and replace $\phi_0( \M_d(\y^\star))$ by $\log \det \M_d(\y^\star)$, for which the gradient is $\M_d(\y^\star)^{-1}$.

Using Table~\ref{tab:gradient} we find that $c^\star\nabla\phi_q(\M_d(\y^\star))=\M_d(\y^\star)^{q-1}$. It follows that
\begin{align*}
c^\star\lambda^\star\overset{\eqref{lambdastar}}{=}c^\star\langle\nabla\phi_q(\M_d(\y^\star)),\M_d(\y^\star)\rangle &=\mathrm{trace}(\M_d(\y^\star)^{q})\\
\text{and}\quad c^\star\langle\nabla\phi_q(\M_d(\y^\star)),\v_d(\x)\v_d(\x)^\top\rangle &\overset{\eqref{christoffel-general}}{=}p_d^\star(x)
\end{align*}

%Therefore $p^\star=\lambda^\star-p^\star_d\in\mathcal{P}_{2d}(\mathcal{X})$.
Therefore, equation \eqref{pstar} is equivalent to $p^\star:=c^\star\, \hat{p}^\star=c^\star\, \lambda^\star-p^\star_d\in\mathcal{P}_{2d}(\mathcal{X})$. To summarize,
\begin{equation*}
	p^\star(x) = \mathrm{trace}(\M_d(\y^\star)^{q}) - p_d^\star(x)\in\mathcal{P}_{2d}(\mathcal{X}).
\end{equation*}
Since the KKT-conditions are necessary and sufficient, the equivalence statement follows.

%Next, the complementarity condition $\langle\y^\star,p^\star\rangle=0$ reads
Finally, we investigate the measure $\mu^\star$ associated with $\y^\star$. Multiplying the complementarity condition $\langle\y^\star,\hat{\mathbf{p}}^\star\rangle=0$ with $c^\star$, we have}
\[
\int_\K \underbrace{p^\star(\x)}_{\geq0\mbox{ on }\K}\,d\mu^\star(\x) = 0.
\]
%which implies that
Hence, the support of $\mu^\star$ is included in the algebraic set $\Omega=\{\x\in\K: p^\star(\x)=0\}$.

The measure $\mu^\star$ is an atomic measure supported on at most $\binom{n+2d}{n}$ points. This follows from Tchakaloff's theorem (see \cite[Theorem B.12]{lasserre} or \cite{bayer2006proof} for instance), which states that for every finite Borel probability measure on $\K$ and every $s\in\N$, there exists an atomic measure $\mu_s$ supported on $\ell\leq\binom{n+s}{n}$ points such that all moments of~$\mu_s$ and $\mu^\star$ agree up to order $s$. For $s=2d$ we get that $\ell\leq \binom{n+2d}{n}$. If $\ell<\binom{n+d}{n}$, then ${\rm rank}\ \M_d(\y^\star) <\binom{n+d}{n}$ in contradiction to $\M_d(\y^\star)$ being non-singular. Therefore, $\binom{n+d}{n}\leq \ell\leq \binom{n+2d}{n}$.

{
\begin{rem} 
The last paragraph has to be adapted as follows in the general case. Recall that there exists a full row rank matrix~$\mathfrak A$ of size ${p\times\binom{n+d}{n}}$ such that the regression polynomials satisfy ${\mathbf F}(x)=\mathfrak A\,\v_d(x)$. Recall also that we are optimizing over the cone of matrices of the form ${\mathbf M}_d(\y):=\mathfrak AM_d(\y)\mathfrak A^{\top}$ indexed by moment sequences $\y$.

First, note that 
\[
\mathrm{rank}\,{\mathbf M}_d(\y)=\min(p, \mathrm{rank}\, M_d(\y))
\]
and recall that the optimal solution ${\mathbf M}_d(\y^\star)$ has full rank, namely it holds that $\mathrm{rank}\,{\mathbf M}_d(\y^\star)=p$. We deduce that $ \mathrm{rank}\, M_d(\y^\star)\geq p$ so that $\mu^\star$ has at least $p$ support points.

Then, consider the vector space spanned by the constant function $1$ and the polynomials $x\mapsto\mathbf{f}_i(x)\mathbf{f}_j(x) $ for $1\leq i,j\leq p$. Denote by $\overline s$ its dimension and observe that 
\[
\overline s\leq \min\Big[1+\frac{p(p+1)}2, \binom{n+2d}{n}\Big]\,.
\] 
The first argument in the minimum is the number of quadratic terms~$\mathbf{f}_i\mathbf{f}_j$ while the second comes from the observation that  their span is included in the vector space of multivariate polynomials of $n$ variables of degree at most~$2d$. Recall that we want to represent the outcome of the linear evaluations
\[
({\mathbf M}_d(\y^\star))_{i,j}=\int\mathbf{f}_i\mathbf{f}_j \mathrm{d}\mu^\star
\,,\quad 1\leq i,j\leq p\,,
\]
by a discrete probability measure $\mu^\star$. By Tchakaloff's theorem, see for instance \cite[Corollary~2{\,}\footnote{In \cite[Corollary 2]{bayer2006proof}, the reader may consider $(\phi_j)_{j=1,\ldots,\overline s}$ any basis of the vector space spanned by the constant function $1$ and the polynomials $x\mapsto\mathbf{f}_i(x)\mathbf{f}_j(x) $ to get the result.}]{bayer2006proof}, we get that~there exists a representing probability measure~$\mu^\star$ of ${\mathbf M}_d(\y^\star)$ with at most $\overline s$ support points.  
\label{rem:chachacha}
\end{rem}
}

\newpage 

{
\section{Numerical results for the Examples}\label{table}

We list in Table \ref{Table} details on the results for the two-dimensional examples (Sections \ref{expl2}, \ref{expl2a}, \ref{expl2b}, and \ref{expl2c}), namely, the numerical values of the support points and their corresponding weights.}

\vfill

\begin{table}[h]
{
	\begin{tabular}{|l|cc|cc|cc|cc|}
		\hline
		     & \multicolumn{2}{|c|}{Wynn} & \multicolumn{2}{|c|}{Ellipses} & \multicolumn{2}{|c|}{Moon} & \multicolumn{2}{|c|}{Folium}\\
		     & $(x_1,x_2)$    & $\omega$  & $(x_1,x_2)$ & $\omega$         & $(x_1,x_2)$ & $\omega$     & $(x_1,x_2)$ & $\omega$\\
		\hline
		$d=1$& (-0.35,-0.35) & 0.125      & (-0.00,-0.75) & 0.250          & (-0.80, 0.00) & 0.329       & ( 0.29,-0.55) & 0.333\\
		     & (-0.35, 0.35) & 0.281      & (-0.90,-0.00) & 0.250          & ( 0.07,-0.53) & 0.305       & (-1.00, 0.00) & 0.333\\
		     & ( 0.35,-0.35) & 0.281      & ( 0.90, 0.00) & 0.250          & ( 0.07, 0.53) & 0.305       & ( 0.29, 0.55) & 0.333\\
		     & ( 0.71, 0.71) & 0.313      & ( 0.00, 0.75) & 0.250          & ( 0.33,-0.29) & 0.031       &&\\
		     &               &            &               &                & ( 0.33, 0.29) & 0.031       &&\\
		\hline
		$d=2$& (-0.35,-0.35) & 0.163      & (-0.45,-0.65) & 0.134          & (-0.39,-0.57) & 0.167       & (-1.00, 0.00) & 0.167\\
		     & (-0.35, 0.35) & 0.165      & (-0.90,-0.00) & 0.139          & (-0.80, 0.00) & 0.167       & (-0.60,-0.21) & 0.166\\
		     & ( 0.12, 0.12) & 0.066      & (-0.00,-0.39) & 0.093          & (-0.20,-0.00) & 0.167       & (-0.60, 0.21) & 0.166\\
		     & ( 0.35,-0.35) & 0.165      & ( 0.45,-0.65) & 0.134          & ( 0.29,-0.35) & 0.167       & ( 0.28,-0.56) & 0.162\\
		     & ( 0.18, 0.53) & 0.141      & (-0.45, 0.65) & 0.134          & (-0.39, 0.57) & 0.167       & ( 0.21,-0.20) & 0.088\\
		     & ( 0.53, 0.18) & 0.141      & ( 0.00, 0.39) & 0.093          & ( 0.29, 0.35) & 0.167       & ( 0.21, 0.20) & 0.088\\
		     & ( 0.71, 0.71) & 0.159      & ( 0.90, 0.00) & 0.139          &               &             & ( 0.28, 0.56) & 0.162\\
		     &               &            & ( 0.45, 0.65) & 0.134          &               &             &&\\
		\hline
		$d=3$& (-0.35,-0.35) & 0.095      & (-0.64,-0.53) & 0.085          & (-0.57,-0.47) & 0.099       & (-1.00,-0.00) & 0.100\\
		     & ( 0.02,-0.35) & 0.074      & (-0.90, 0.00) & 0.088          & (-0.08,-0.59) & 0.098       & (-0.77,-0.20) & 0.099\\
		     & (-0.35, 0.02) & 0.074      & (-0.00,-0.75) & 0.088          & (-0.80, 0.00) & 0.100       & (-0.77, 0.20) & 0.099\\
		     & ( 0.35,-0.35) & 0.096      & (-0.36,-0.32) & 0.075          & (-0.45,-0.18) & 0.061       & (-0.45, 0.00) & 0.077\\
		     & ( 0.14,-0.12) & 0.044      & ( 0.00,-0.39) & 0.005          & (-0.11,-0.30) & 0.062       & (-0.14,-0.00) & 0.033\\
		     & (-0.12, 0.14) & 0.044      & (-0.64, 0.53) & 0.085          & (-0.45, 0.18) & 0.061       & ( 0.10,-0.41) & 0.098\\
		     & (-0.35, 0.35) & 0.097      & (-0.36, 0.32) & 0.075          & ( 0.33,-0.29) & 0.099       & ( 0.29,-0.56) & 0.099\\
		     & ( 0.45,-0.06) & 0.088      & ( 0.36,-0.32) & 0.075          & (-0.57, 0.47) & 0.099       & ( 0.31,-0.35) & 0.100\\
		     & (-0.06, 0.45) & 0.088      & ( 0.64,-0.53) & 0.085          & ( 0.11,-0.00) & 0.063       & ( 0.10, 0.41) & 0.098\\
		     & ( 0.39, 0.39) & 0.037      & (-0.00, 0.39) & 0.005          & (-0.11, 0.30) & 0.062       & ( 0.31, 0.35) & 0.100\\
		     & ( 0.61, 0.41) & 0.084      & ( 0.36, 0.32) & 0.075          & (-0.08, 0.59) & 0.098       & ( 0.29, 0.56) & 0.099\\
		     & ( 0.41, 0.61) & 0.084      & (-0.00, 0.75) & 0.088          & ( 0.33, 0.29) & 0.099       &               &      \\
		     & ( 0.71, 0.71) & 0.097      & ( 0.90,-0.00) & 0.088          &               &             &               &      \\
		     &               &            & ( 0.64, 0.53) & 0.085          &               &             &               &      \\
		\hline
	\end{tabular}
	\caption{The numerical values for Examples \ref{expl2}, \ref{expl2a}, \ref{expl2b}, and \ref{expl2c} for the support points $x_i=(x_{i,1},x_{i,2})$ and their corresponding weights $\omega_i$, $i=1,\dotsc,\ell$.}
	\label{Table}
	}
\end{table}

\vfill

\newpage 

\bibliographystyle{abbrv}
\bibliography{biblio}

\end{document}